\documentclass[10pt,reqno,a4paper]{amsart}

\usepackage{amssymb,amsrefs,amscd,amsmath, fullpage}
\usepackage{graphicx,nicefrac}
\usepackage[shortlabels]{enumitem}
\usepackage{mathrsfs}
\usepackage{hyperref}
\usepackage{ourbib}
\usepackage[all]{xy}
\usepackage{callouts}
\usepackage{tikz-cd}

\setlist{leftmargin=8mm}

\hypersetup{
    colorlinks,
    linkcolor={red!50!black},
    citecolor={blue!50!black},
    urlcolor={blue!80!black}
}

\newcommand{\tr}{\mathrm{tr}}

\newcommand{\dM}{{\partial M}}

\newcommand{\scal}{\mathrm{scal}}

\newcommand{\myicon}{$\,\,\,\triangleright$}

\DeclareMathOperator{\id}{\mathrm{id}}
\DeclareMathOperator{\Diff}{\mathrm{Diff}} 
\DeclareMathOperator{\cl}{\mathrm{cl}}
\DeclareMathOperator{\supp}{\mathrm{supp}}
\DeclareMathOperator{\im}{\mathrm{im}}
\DeclareMathOperator{\Int}{\mathrm{int}}
\DeclareMathOperator{\tw}{\mathrm{tw}}

\newtheorem{theorem}{Theorem}[section]
\newtheorem*{theorem*}{Sample Theorem}
\newtheorem{lemma}[theorem]{Lemma}
\newtheorem{proposition}[theorem]{Proposition}
\newtheorem{corollary}[theorem]{Corollary} 
 
\newtheorem*{applications}{Applications}
\newtheorem*{Gerochconjecture}{On the generalised Geroch conjecture with boundary}
\newtheorem*{Gerochconjectureunbdd}{Generalised Geroch conjecture}
\newtheorem*{factC1}{Fact CN1}
\newtheorem*{factC2}{Fact CN2}
\newtheorem*{factC1'}{Fact CN1'}
\newtheorem*{factC2'}{Fact CN2'}

\theoremstyle{definition}
\newtheorem{remark}[theorem]{Remark}
\newtheorem{definition}[theorem]{Definition}

\newcommand{\smallsquare}{     
     \vcenter{\hbox{\scalebox{0.55}{$\;\mathbin{ \square }\;$}}}    
}

\newcommand{\smallblacksquare}{     
     \vcenter{\hbox{\scalebox{0.55}{$\;\mathbin{ \blacksquare }\;$}}}    
}

\begin{document}

\title{Scalar curvature deformations with non-compact boundaries} 
\author{Helge Frerichs}
\address{Universit\"at Augsburg, Institut f\"ur Mathematik, 86135 Augsburg, Germany}
\email{\href{mailto:helge.frerichs@math.uni-augsburg.de}{helge.frerichs@math.uni-augsburg.de}}

\begin{abstract} 
We develop a general deformation principle for families of Riemannian metrics on smooth manifolds with possibly non-compact boundary, preserving lower scalar curvature bounds.
The principle is used in order to strengthen boundary conditions, from mean convex to totally geodesic or doubling.
The deformation principle preserves further geometric properties such as completeness and a given quasi-isometry type.

As an application, we prove non-existence results for Riemannian metrics with (uniformly) positive scalar curvature and mean convex boundary, including progress on a generalised Geroch conjecture with boundary and some investigation of the Whitehead manifold.
\end{abstract}



\thanks{\emph{Acknowledgment.} 
This work is part of my ongoing PhD project under the guidance of Bernhard Hanke.
It is supported by the TopMath Program from the Elite Network of Bavaria.
I am grateful to Shmuel Weinberger for advice regarding the Whitehead manifold.}

\date{}

\maketitle

\section{Introduction}
Existence results for Riemannian metrics with lower scalar curvature bounds on smooth manifolds and lower mean curvature bounds along the boundary have been a topic of interest in differential geometry.
The paper at hand presents a general deformation principle for Riemannian metrics with a lower scalar curvature bound on manifolds with possibly non-compact boundary.
It is motivated by a classical flexibility result of Gromov-Lawson \cite[Thm.~5.7]{GL1980}, refined by Almeida \cite[Thm.~1.1]{Almeida} and placed in a new framework by Bär-Hanke \cite[Thm.~3.7]{BH2023}.
All these results deal with compact boundaries.
The principle to be presented allows to compare various boundary conditions and applies whenever one needs a deformation which is mean curvature non-increasing.
A special case of our main result is the following:

\begin{theorem*}(See Theorem~\ref{theorem27})
Let $n\geq 2$.
Let $M$ be an $n$-dimensional manifold with non-empty (possibly non-compact) boundary and let $g$ be a Riemannian metric on $M$ with $\scal_g>0$.
Let $k\in C^{\infty}(\dM;T^\ast\dM\otimes T^\ast\dM)$ be a symmetric $(0,2)$-tensor field along $\dM$ satisfying $\tfrac{1}{n-1}\tr_{g_0}(k)\leq H_g$.
Here $g_0$ denotes the metric induced on $\dM$ and $H_g$ denotes the mean curvature of the boundary.

Then there exists a continuous path $(f_s)_{s\in[0,1]}$ of Riemannian metrics on $M$ such that
\begin{enumerate}
\item[(a)]{$\scal_{f_s}>0$ for all $s\in[0,1]$;}
\item[(b)]{$f_0=g$;}
\item[(c)]{$\mathrm{II}_{f_1}=k$ where $\mathrm{II}$ denotes the second fundamental form of the boundary.}
\end{enumerate}
The space of Riemannian metrics on $M$ is equipped with the weak $C^\infty$-topology.
\end{theorem*}
We can add various features to this assertion:
\begin{enumerate}
\item[\myicon]{Instead of only considering positive scalar curvature (psc), an arbitrary continuous function $\sigma\colon M\to\mathbb{R}$ can be used as a lower scalar curvature bound.
This means that if $\scal_g>\sigma$, then the metrics $f_s$ can be constructed in such a way that $\scal_{f_s}>\sigma$ for all $s\in[0,1]$.}
\item[\myicon]{The result also applies to families of metrics and in a relative version.
Similar to the work of Bär-Hanke \cite{BH2023}, the last aspect is a main advantage of our deformation principle compared to Gromov-Lawson's \cite{GL1980}:
It allows not only to answer existence questions for metrics with boundary conditions, but also to treat spaces of such metrics and to study their homotopy type, compare Bär-Hanke \cite[Cor.~4.9]{BH2023} and \cite[Cor.~4.14]{BH2023}.}
\item[\myicon]{The deformations may be supported in an arbitrarily small neighbourhood of the boundary.}
\item[\myicon]{Our method preserves completeness of metrics, i.e. if $g(\xi)$ is complete for some $\xi\in K$, then $f(\xi,s)$ is complete for all $s\in[0,1]$.}
\item[\myicon]{Recall that two Riemannian metrics $g$ and $h$ on $M$ are called \emph{quasi-isometric via the identity} if there exist constants $A\geq 1$ and $B\geq 0$ such that
\begin{equation}\nonumber
\frac{1}{A}\,d_h(p,q)-B\leq d_g(p,q)\leq A\,d_h(p,q)+B\quad\text{for all $p,q\in M$.}
\end{equation}
Here $d_g$ and $d_h$ denote the Riemannian distances associated to $g$ and $h$.
The above condition defines an equivalence relation on the space of Riemannian metrics on $M$.

The deformation principle preserves a given quasi-isometry type, meaning that it is carried out within a fixed equivalence class of metrics.}
\end{enumerate}

A boundary condition of high interest is what we call the doubling property for metrics on a manifold with boundary.
A \textit{doubling metric} is defined as a metric whose twofold copy is a smooth Riemannian metric on the double of the manifold, see Definition \ref{defdoubling}.
By working with doubling metrics, our deformation principle makes a connection between scalar curvature geometry on manifolds with boundary and on manifolds without boundary.
This strategy has already been employed by Gromov-Lawson \cite{GL1980}, Almeida \cite{Almeida} and Carlotto-Li \cite{CL20241,CL20242}, for instance.

In a similar spirit, using work of Hanke-Kotschick-Roe-Schick \cite[Cor.~1.8]{HKRS}, Cecchini \cite[Thm.~C]{Cecchini} and Chang-Weinberger-Yu \cite[Thm.~1]{CWY}, we prove some non-existence results for (uniform) psc-metrics with mean convex boundary.
The term `mean convex' indicates that the mean curvature is non-negative at each boundary point.
We state our applications from Corollary \ref{R2}, \ref{fromHKRS08}\hspace{1pt}-\hspace{1pt}\ref{Application3} and \ref{Whitehead}.
\begin{applications}
\begin{enumerate}
\item[(1)]{(Corollary~\ref{R2}) The right half-plane $\mathbb{R}^2_{x_1\geq 0}$ cannot carry any complete Riemannian metric of uniformly positive scalar curvature which has mean convex boundary.}
\item[(2)]{(Corollary~\ref{fromHKRS08}) Let $n\geq 3$.
The right half-space $\mathbb{R}^n_{x_1\geq 0}$ cannot carry any complete Riemannian metric of uniformly positive scalar curvature which has mean convex boundary and which is quasi-isometric to the Euclidean metric via the identity.}
\item[(3)]{(Corollary~\ref{Application3}) Let $N:=\mathbb{R}^n_{x_1\geq 0}\#T^n, n\geq 2$ be a connected sum of the right half-space with an $n$-torus attached in the interior.
Then $N$ does not admit a complete Riemannian metric of positive scalar curvature which has mean convex boundary.}
\item[(4)]{(See Corollary~\ref{Whitehead}) The Whitehead manifold has contractible submanifolds with non-compact boundary which do not admit complete Riemannian metrics of uniformly positive scalar curvature with mean convex boundary.}
\end{enumerate}
\end{applications}
Another application is Corollary \ref{Geroch}.
Here we refer to the famous Geroch conjecture, which states that the torus does not admit a metric of positive scalar curvature and that every metric on the torus with non-negative scalar curvature is flat.
It has been independently proven by Schoen-Yau and Gromov-Lawson.

A subsequent conjecture involves connected sums of tori with arbitrary manifolds.
This is called the `generalised Geroch conjecture' or the `Geroch conjecture with arbitrary ends'.
So far, it has been proven for $n=3$ by Lesourd-Unger-Yau \cite[Thm.~1.2]{LUY} and for all $n\leq 7$ by Chodosh-Li \cite[Thm.~3]{CL}.
Wang-Zhang \cite[Thm.~1.1]{WZ} proved the conjecture in all dimensions $n\in\mathbb{N}$ for the case where $M$ is spin.

\begin{Gerochconjectureunbdd}
Let $n\in\mathbb{N}$.
For any $n$-manifold $M$ (without boundary), the connected sum $T^n\# M$ does not admit a complete Riemannian metric of positive scalar curvature.
Every complete metric of non-negative scalar curvature is flat.
\end{Gerochconjectureunbdd}

We contribute to a version of Geroch's conjecture with boundary.
\begin{Gerochconjecture}(Corollary \ref{Geroch} and \ref{Gerochadd})
Let $n\geq 2$ such that the generalised Geroch conjecture (without boundary) holds true for all manifolds of dimension $n$.
Then for any $n$-manifold $M$ with non-empty boundary, the connected sum $M\# T^n$ with a torus attached in the interior does not admit a complete Riemannian metric of positive scalar curvature which has mean convex boundary.
\end{Gerochconjecture}

Together with a recent result of Cruz-Silva Santos \cite[Thm.~1.1]{CSS} and three splitting theorems for manifolds with boundary by Kasue \cite[Thm.~B~\&~C]{Kasue} and Sakurai \cite[Thm.~1.8]{Sakurai}, we can draw the following conclusions for all $n\geq 3$ for which the generalised Geroch conjecture holds:
\begin{enumerate}
\item[\textit{(a)}]{\textit{Every complete Riemannian metric on $M\# T^n$ with non-negative scalar curvature and mean convex boundary is Ricci-flat.}}
\item[\textit{(b)}]{\textit{Let $M$ be a connected $n$-manifold with non-empty boundary and one of the following topological properties:
\begin{enumerate}\nonumber
\item[(i)]{$M$ is non-compact and $\partial M$ is compact or}
\item[(ii)]{$\partial M$ is disconnected and it has a compact connected component.}
\end{enumerate}
Then the connected sum $M\# T^n$ with a torus attached in the interior does not admit a complete Riemannian metric of non-negative scalar curvature which has mean convex boundary.}}
\item[\textit{(c)}]{\textit{For any connected $n$-manifold $M$ with non-empty and non-compact boundary, the connected sum $M\# T^n$ with a torus attached in the interior does not admit a complete Riemannian metric $g$ of non-negative scalar curvature which has mean convex boundary and which has a unit-speed normal geodesic $\gamma\colon[0,\infty)\to M$ emanating from the boundary with
\begin{equation}\nonumber
\sup\{t\geq 0: d_g(\gamma(t),\partial M)=t\}=\infty.
\end{equation}}}
\end{enumerate}
The paper is organised as follows:
The first section discusses the geometry of doubling metrics. 
We will spend some effort to prove the following seemingly tautological result, see Theorem~\ref{completemetrics}:
If $g$ is a complete doubling metric on a manifold with boundary, then the twofold copy $g\cup g$ is a complete metric on the double of the manifold.
This result is of independent interest.
The proof makes use of absolutely continuous curves on manifolds with boundary, relying on work of Burtscher \cite{Burtscher}.

In the second section, we develop the main deformation principle for families of metrics with lower scalar curvature bounds.
The essential part is to generalise the deformation scheme \cite[Prop.~3.3 and 3.6]{BH2023} of Bär-Hanke to non-compact boundaries.
Here we make use of the local flexibility lemma in \cite{BH2022}.
Our deformation principle combines the work of Bär-Hanke with a closer study of the topology of the boundary.
In the last section we present our non-existence results that have been mentioned above.

\section{Completeness of doubling metrics} 
Let $M$ be a smooth manifold with non-empty boundary.
Given a Riemannian metric $g$ on $M$, its normal exponential map provides a collar neighbourhood
\begin{equation}\nonumber
\varrho_g\colon V\to U\subset M\,,\,\varrho_g(t,p)=\exp_p(tN_g(p))
\end{equation}
where $V\subset[0,\infty)\times\dM$ is an open neighbourhood of $\{0\}\times\dM$ and $N_g$ is the inward-pointing unit normal vector field along $\dM$.
We call $\varrho_g$ a \textit{geodesic collar neighbourhood}.
It induces a smooth structure on the topological double $\mathsf{D}M=M\cup_{\dM}M$.

The resulting smooth manifold is sometimes referred to as $\mathsf{D}M^g$ to emphasise the smooth structure.
Note that two different Riemannian metrics induce diffeomorphic smooth structures on $\mathsf{D}M$, since two arbitrary collar neighbourhoods for $\dM$ are isotopic.
However, the two smooth structures are not the same in general.
This will be important later.

\begin{definition}\label{defdoubling}
We call $g$ a \emph{doubling metric} if the twofold copy $g\cup g$ is a smooth Riemannian metric on $\mathsf{D}M^g$.
\end{definition}
There is another characterisation for the doubling property:
According to the generalised Gauss lemma, pulling back $g$ along $\varrho_g$ yields a generalised cylinder metric
\begin{equation}\nonumber
g=\mathrm{d}t^2+g_t
\end{equation}
where $g_t$ is given by $g_\bullet\colon V\to T^\ast\dM\otimes T^\ast\dM,g_t(p)(v,w)=(\varrho_g^\ast g)(t,p)(v,w)$.
The family $g_\bullet$ gives rise to new families $\dot{g}_\bullet,\ddot{g}_\bullet,g^{(\ell)}_\bullet\colon V\to T^\ast\dM\otimes T^\ast\dM$ of $(0,2)$-tensor fields on $\dM$ defined by
\begin{align}\nonumber
\dot{g}_t(p)(v,w)&:=\frac{\mathrm{d}}{\mathrm{d}t}(g_t(p)(v,w)),\\\nonumber
\ddot{g}_t(p)(v,w)&:=\frac{\mathrm{d}^2}{\mathrm{d}t^2}(g_t(p)(v,w)),\\\nonumber
g_t^{(\ell)}(p)(v,w)&:=\frac{\mathrm{d}^{(\ell)}}{\mathrm{d}t^{(\ell)}}(g_t(p)(v,w))
\end{align}
for all $(t,p)\in V$ and $v,w\in T_p\dM$.
Then $g$ is doubling iff $g_0^{(\ell)}\equiv 0$ for $\ell$ odd.

Apart from that, the second fundamental form of the boundary is given by $\mathrm{II}_g=-\tfrac{1}{2}\dot{g}_0$, see \cite[Prop.~4.1]{BGM}.
This means that doubling metrics need to have a totally geodesic boundary.
The aim of this section is to prove the following theorem for a connected $M$.
Related questions have been studied by Pigola-Veronelli \cite{PV}.
\begin{theorem}\label{completemetrics}
If $g$ is a complete doubling metric on $M$, then $G:=g\cup g$ is a complete metric on $\mathsf{D}M$.
\end{theorem}
By completeness we mean metric completeness, which is the convergence of all Cauchy sequences with respect to the Riemannian distance.
The next equation is crucial:
\begin{proposition}\label{samedistance}
It holds $d_g(p,q)=d_G(p,q)$ for all $p,q\in M$.
\end{proposition}
Here $d_g(p,q)$ is the Riemannian distance between $p$ and $q$ in $M$ with respect to $g$ and $d_G(p,q)$ is their distance in $\mathsf{D}M$ with respect to $G$.
Note that we have to consider all admissible curves in $\mathsf{D}M$ in order to determine $d_G(p,q)$, whereas for $d_g(p,q)$ we only look at curves in $M$.
This is why the inequality $d_g(p,q)\geq d_G(p,q)$ is obvious.

To address the other inequality, we use a reflection argument: Reflecting an admissible curve in $\mathsf{D}M$ along the hypersurface $\dM$ results in a curve in $M$ that retains its original length.
Therefore, there are no shorter curves between $p$ and $q$ in $\mathsf{D}M$ than in $M$.

Unfortunately, not every class of admissible curves is invariant under reflection.
For example, the common class of piecewise smooth curves is not, as can be seen from the right half-plane $M=\mathbb{R}^2_{x_1\geq 0}$ and the Euclidean metric $g=g_{\mathrm{eucl}}$.
Its double $\mathsf{D}M^g$ is diffeomorphic to $\mathbb{R}^2$ via the canonical map
\begin{equation}\nonumber
\kappa\colon\mathbb{R}^2\to\mathsf{D}M\,,\,\kappa(x)=\left\{\begin{array}{ll} \left[(x_1,x_2),1\right], & x_1\geq 0 \\
         \left[(-x_1,x_2),2\right], & x_1\leq 0\end{array}\right..
\end{equation}
Then
\begin{equation}\nonumber
\gamma\colon[0,1]\to\mathbb{R}^2\,,\,\gamma(t)=\left\{\begin{array}{ll}(e^{-1/t^2}\sin(\tfrac{\pi}{t}),t), & t\neq 0 \\
         (0,0), & t=0\end{array}\right.
\end{equation}
is a piecewise smooth curve but its reflection to $M$ is not.
This is why we consider the larger class of absolutely continuous curves.
The following considerations are based on Burtscher \cite{Burtscher}.
\begin{definition}
Let $I\subset\mathbb{R}$ be an interval (open or closed) or an open subset.
A function $f\colon I\to\mathbb{R}^n$ is called \emph{absolutely continuous} if for every $\varepsilon>0$, there exists some $\delta>0$ such that for every $m\in\mathbb{N}$ and every family $\{(a_i,b_i)\}_{i=1}^m$ of disjoint intervals with $[a_i,b_i]\subset I$ and total length $\sum_{i=1}^m|b_i-a_i|<\delta$, it holds
\begin{equation}\nonumber
\sum_{i=1}^m\Vert f(b_i)-f(a_i)\Vert<\varepsilon.
\end{equation}
Here $\Vert\cdot\Vert$ denotes the Euclidean norm on $\mathbb{R}^n$.

A function $f\colon I\to\mathbb{R}^n$ is said to be \emph{locally absolutely continuous} if it is absolutely continuous on every closed interval $[a,b]\subset I$.

A function $f\colon I\to\mathbb{R}^n_{x_1\geq 0}$ is called (locally) absolutely continuous if it is (locally) absolutely continuous  as a function $I\to\mathbb{R}^n$.
\end{definition}
Now let $M$ be a connected smooth manifold with boundary of dimension $n\in\mathbb{N}$.
The boundary may or may not be empty.
\begin{definition}
Let $I\subset\mathbb{R}$ be a closed interval.
A continuous curve $\gamma\colon I\to M$ is called absolutely continuous if, for every chart $(\varphi,U)$ of $M$, the function
\begin{equation}\nonumber
\varphi\circ\gamma\colon\gamma^{-1}(U)\to\varphi(U)\subset\mathbb{R}^n_{x_1\geq 0}
\end{equation}
is locally absolutely continuous.
The set of all absolutely continuous curves $\gamma\colon[0,1]\to M$ is denoted by $\mathscr{A}_{ac}(M)$ respectively $\mathscr{A}_{ac}$.
For any two points $p,q\in M$, we denote by $\mathscr{A}_{ac}^{p,q}\subset\mathscr{A}_{ac}$ the set of all absolutely continuous curves $\gamma\colon[0,1]\to M$ with $\gamma(0)=p$ and $\gamma(1)=q$.
\end{definition}
We show that absolute continuity is a well-behaved concept in differential topology.
\begin{lemma}\label{oneatlas}
Let $I\subset\mathbb{R}$ be a closed interval.
A continuous curve $\gamma\colon I\to M$ is absolutely continuous iff for every $p\in\gamma(I)$, there exists a local chart $(\varphi,U)$ of $M$ around $p$ such that $\varphi\circ\gamma\colon\gamma^{-1}(U)\to\varphi(U)$ is locally absolutely continuous.
\end{lemma}
\begin{proof}
Let $(\psi,V)$ be an arbitrary local chart of $M$ with $V\cap\gamma(I)\neq\varnothing$.
Let $[a,b]\subset\gamma^{-1}(V)$ be a closed interval.
We need to prove absolute continuity for $\psi\circ\gamma\colon[a,b]\to\psi(V)$.

By assumption there exists an atlas $(\varphi_\ell,U_\ell)_{\ell\in I}$ of $M$ such that $\varphi_\ell\circ\gamma\colon\gamma^{-1}(U_\ell)\to\varphi_\ell(U_\ell)$ is locally absolutely continuous for every $\ell\in I$.
Then we find a partition $a=s_0<s_1<\dots<s_N=b$ of $[a,b]$ so that $\gamma|_{[s_{j-1},s_j]}$ is entirely contained in one of the neighbourhoods $U_\ell$.
To every $1\leq j\leq N$ we assign such a neighbourhood $U_j$.

Every transition map $\psi\circ\varphi_j^{-1}\colon\varphi_j(V\cap U_j)\to\psi(V\cap U_j)$ is smooth, hence Lipschitz continuous on the compact subspace $\im(\varphi_j\circ\gamma|_{[s_{j-1},s_j]})$.
Let $(L_j)_{j=1}^N$ be a family of associated positive Lipschitz constants.

Now we prove that $\psi\circ\gamma\colon[a,b]\to\psi(V)$ is absolutely continuous.
To this end, let $\varepsilon>0$ and let $\delta>0$ be small enough so that for every $1\leq j\leq N$, every $m\in\mathbb{N}$ and every family $\{(a_i,b_i)\}_{i=1}^m$ of disjoint intervals with $[a_i,b_i]\subset[s_{j-1},s_j]$ and total length $\sum_{i=1}^m|b_i-a_i|<\delta$, it holds
\begin{equation}\nonumber
\sum_{i=1}^m\Vert(\varphi_j\circ\gamma)(b_i)-(\varphi_j\circ\gamma)(a_i)\Vert<\frac{\varepsilon}{L_jN}.
\end{equation}
Let $m\in\mathbb{N}$ and let $\{(a_i,b_i)\}_{i=1}^m$ be a family of disjoint intervals with $[a_i,b_i]\subset[a,b]$ and total length $\sum_{i=1}^m|b_i-a_i|<\delta$.
After dividing the intervals $(a_i,b_i)$ if necessary, we can assume that for each $1\leq i\leq m$, there is some $1\leq j\leq N$ with $[a_i,b_i]\subset[s_{j-1},s_j]$.
Then it holds
\begin{align}\nonumber
\sum_{i=1}^m&\Vert(\psi\circ\gamma)(b_i)-(\psi\circ\gamma)(a_i)\Vert\\\nonumber
&=\sum_{j=1}^N\sum\limits_{\substack{[a_i,b_i]\\\subset[s_{j-1},s_j]}}\Vert(\psi\circ\gamma)(b_i)-(\psi\circ\gamma)(a_i)\Vert\\\nonumber
&=\sum_{j=1}^N\sum\limits_{\substack{[a_i,b_i]\\\subset[s_{j-1},s_j]}}\Vert(\psi\circ\varphi_j^{-1})((\varphi_j\circ\gamma)(b_i))-(\psi\circ\varphi_j^{-1})((\varphi_j\circ\gamma)(a_i))\Vert\\\nonumber
&\leq\sum_{j=1}^N\sum\limits_{\substack{[a_i,b_i]\\\subset[s_{j-1},s_j]}}L_j\cdot\Vert(\varphi_j\circ\gamma)(b_i)-(\varphi_j\circ\gamma)(a_i)\Vert<\sum_{j=1}^NL_j\cdot\frac{\varepsilon}{L_jN}=\varepsilon.\qedhere
\end{align}
\end{proof}
\begin{remark}
Burtscher \cite[Prop.~3.4]{Burtscher} is a consequence of Lemma~\ref{oneatlas}.
\end{remark}
\begin{proposition}\label{aedifferentiable}
Let $I\subset\mathbb{R}$ be a closed interval. 
Every absolutely continuous curve $\gamma\colon I\to M$ is differentiable almost everywhere.
If $g$ is a Riemannian metric on $M$, then $\vert\dot{\gamma}\vert_g\colon I\to\mathbb{R}$ is integrable.
\end{proposition}
\begin{proof}
See \cite[Prop 3.7]{Burtscher}.
\end{proof}
We fix a Riemannian metric $g$ on $M$.
By Proposition~\ref{aedifferentiable} each absolutely continuous curve $\gamma$ has a well-defined length $\ell_g(\gamma):=\int_I\vert\dot{\gamma}(s)\vert_g\;\mathrm{d}s$.
Then $(\mathscr{A}_{ac},\ell_g)$ is a length structure on $M$ and thus induces a distance function $d_{ac}$.

We want to show that $d_{ac}$ coincides with the usual Riemannian distance $d_g$ which is determined by piecewise smooth curves. 
Note that piecewise smooth curves are absolutely continuous since they are absolutely continuous on a partition of their domain.

\begin{definition}
The \emph{variational metric} on $\mathscr{A}_{ac}$ is defined by
\begin{equation}\nonumber
D_{ac}\colon\mathscr{A}_{ac}\times\mathscr{A}_{ac}\to\mathbb{R}\,,\,D_{ac}(\gamma_1,\gamma_2)=\sup_{s\in[0,1]}d_g(\gamma_1(s),\gamma_2(s))+\int_0^1|\vert\dot{\gamma}_1(s)\vert_g-\vert\dot{\gamma}_2(s)\vert_g|\;\mathrm{d}s.
\end{equation}
\end{definition}
This is indeed a metric, i.e. $(\mathscr{A}_{ac},D_{ac})$ is a metric space.
\begin{remark}\label{lengthlipschitz}
The length functional $\ell_g\colon\mathscr{A}_{ac}\to\mathbb{R}$ is 1-Lipschitz with respect to the variational metric $D_{ac}$.
\end{remark}
\begin{proposition}\label{dense}
For every $p,q\in M$, the set $\mathscr{A}_{\infty}^{p,q}$ of piecewise smooth curves $\gamma\colon[0,1]\to M$ with $\gamma(0)=p$ and $\gamma(1)=q$ is dense in $(\mathscr{A}_{ac}^{p,q},D_{ac})$.
\end{proposition}
\begin{proof}
See \cite[Thm.~3.11]{Burtscher}.
The proof also applies in the case where $M$ has non-empty boundary.
The only thing to notice is the following:
If $\gamma\colon[0,1]\to\mathbb{R}^n_{x_1\geq 0}$ is a curve in the right half-plane, then so is the convolution $\rho\ast\gamma$ with a Friedrichs mollifier $\rho$, where the convolution is componentwise.
\end{proof}
\begin{remark}
Given an absolutely continuous curve $\gamma\colon[0,1]\to M$ and $\varepsilon>0$, we find a $\delta>0$ such that $\int_{s_1}^{s_2}\vert\dot{\gamma}\vert_g<\varepsilon$ for all $s_1,s_2\in[0,1]$ with $|s_1-s_2|<\delta$.
According to Burtscher, this is due to the fundamental theorem of calculus for absolutely continuous functions.
In fact, the statement holds by integrability of $\vert\dot{\gamma}\vert_g$ and the dominated convergence theorem.

However, the fundamental theorem of calculus can be used at another point:
If $f\colon\mathbb{R}\to\mathbb{R}$ is an absolutely continuous and bounded function, one can show that the convolution $f\ast\rho$ with a Friedrichs mollifier $\rho$ is smooth and that its derivative is given by $(f\ast\rho)'=f\ast\rho'$.

Then choose $a,b\in\mathbb{R}$ with $\supp\rho\subset[a,b]$ and let $s\in\mathbb{R}$.
By the fundamental theorem of calculus for absolutely continuous functions, we find
\begin{align}\nonumber
(f\ast&\rho')(s)-(f'\ast\rho)(s)\\\nonumber
&=\int_a^bf(s-\sigma)\rho'(\sigma)-f'(s-\sigma)\rho(\sigma)\;\mathrm{d}\sigma\\\nonumber
&=\int_a^b\frac{\mathrm{d}}{\mathrm{d}\sigma}(f(s-\sigma)\rho(\sigma))\;\mathrm{d}\sigma\\\nonumber
&=f(s-b)\rho(b)-f(s-a)\rho(a)=0
\end{align}
respectively $(f\ast\rho')(s)=(f'\ast\rho)(s)$.
Therefore $(f\ast\rho)'=f'\ast\rho$.
This fact is used in equation (3.5b) of \cite[Thm.~3.11]{Burtscher}.
\end{remark}
By Remark~\ref{lengthlipschitz} and Proposition~\ref{dense}, we obtain
\begin{corollary}\label{dacdg}
$d_{ac}=d_g$.
\end{corollary}
The next step is to perform the reflection argument for absolutely continuous curves. 
From now on, $g$ is a doubling metric on a manifold $M$ with non-empty boundary.
Put $G:=g\cup g$ as a metric on $\mathsf{D}M^g$.
\begin{proposition}\label{Reflectionargument}
Let $I\subset\mathbb{R}$ be a closed interval and let $\gamma\colon I\to\mathsf{D}M^g$ be an absolutely continuous curve.
\begin{enumerate}
\item[(a)]{Then $\tilde{\gamma}\colon I\to M,\tilde{\gamma}:=\mathrm{pr}\circ\gamma$ is an absolutely continuous curve, where $\mathrm{pr}\colon\mathsf{D}M\to M$ denotes the projection map onto a fixed copy of $M$.}
\item[(b)]{It holds $\vert\dot{\gamma}(s)\vert_G=\vert\dot{\tilde{\gamma}}(s)\vert_g$ for all $s\in I$, where both $\gamma$ and $\tilde{\gamma}$ are differentiable.}
\end{enumerate}
\end{proposition}
\begin{proof}
The two copies of $M$ will be denoted by $[M,1]$ and $[M,2]$.
For the proof, let $\mathrm{pr}\colon\mathsf{D}M\to[M,1]$ be the projection onto the first copy.

\textit{Regarding (a):}
Let $p\in\tilde{\gamma}(I)$.
If $p\notin\dM$ we find a local chart $(\varphi,U)$ of $M$ around $p$ with $U\subset\mathrm{int}(M)$.
Let $[a,b]\subset\tilde{\gamma}^{-1}(U)$.
Since $\gamma|_{[a,b]}$ does not hit the boundary, it either holds $\im(\gamma|_{[a,b]})\subset[U,1]$ or $\im\gamma|_{[a,b]}\subset[U,2]$ due to continuity.
In any case, $\varphi\circ\tilde{\gamma}|_{[a,b]}=\varphi\circ\gamma|_{[a,b]}$ is absolutely continuous.

In the following we consider $p\in\dM$: The hypersurface $\dM\subset\mathsf{D}M$ has a geodesic tubular neighbourhood
\begin{equation}\nonumber
\varrho_G\colon V\to U^G\,,\,\varrho_G(t,p)=\exp_p^G(tN_g(p))
\end{equation}
where $V\subset\mathbb{R}\times\dM$ is an open neighbourhood of $\{0\}\times\dM$ and $N_g$ is the inward pointing unit normal vector field along the boundary of $[M,1]$ with respect to $g$.
Moreover, $V$ is supposed to be symmetric under reflection along $\{0\}\times\dM$.

We have $\varrho_G(V\cap\{t\geq 0\})\subset[M,1]$ and $\varrho_G(V\cap\{t\leq 0\})\subset[M,2]$, as well as
\begin{equation}\label{reflectionequation}
\varrho_G(t,p)=(\mathrm{pr}\circ\varrho_G)(-t,p)
\end{equation}
for all $(t,p)\in V$ with $t\geq 0$ by symmetry of $G$.
Put $U^g:=U^G\cap[M,1]$.
Then
\begin{equation}\nonumber
\varrho_g\colon V\cap\{t\geq 0\}\to U^g\,,\,\varrho_g(t,p):=\varrho_G(t,p)
\end{equation}
is a geodesic collar neighbourhood of $\dM$ in $[M,1]$, see Figure~\ref{Collarsntubes}.
It holds $\tilde{\gamma}^{-1}(U^g)=\gamma^{-1}(U^G)$.
\begin{figure}[h]
\begin{tikzpicture}
\node{\includegraphics[width=0.5\textwidth]{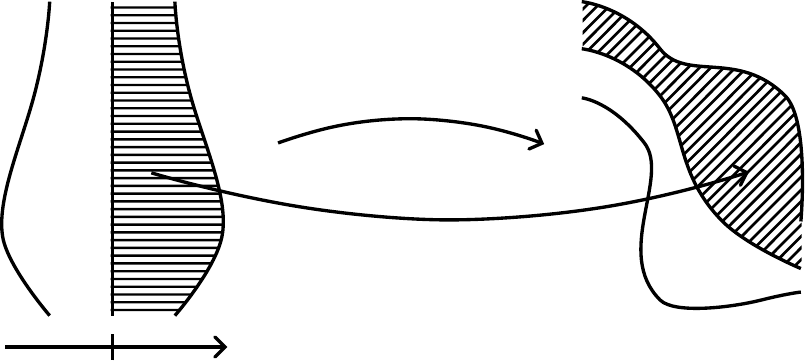}};
\node at (-2.86,-2) {$0$};
\node at (-1.85,-2) {$t$};
\node at (-2.82,2.2) {$V$};
\node at (3,2.2) {$U^G$};
\node at (1.35,1.32) {$\partial M$};
\node at (0.425,-0.7) {$\varrho_g$};
\node at (0.175,0.875) {$\varrho_G$};
\node at (3.7,1.375) {$U^g$};
\end{tikzpicture}
\caption{Tubular and collar neighbourhoods of $\partial M$}
\label{Collarsntubes}
\end{figure}
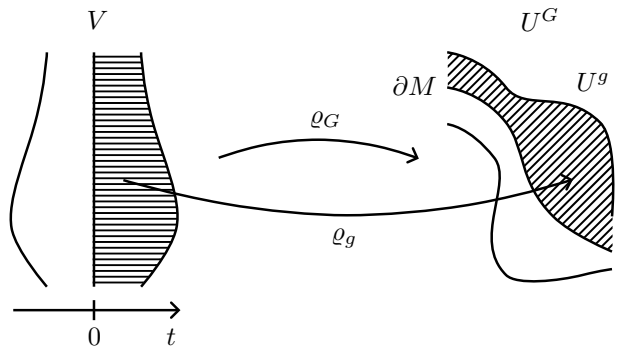

Within the tubular neighbourhood, $\gamma$ is given by
\begin{align}\nonumber
\varrho_G^{-1}\circ\gamma\colon&\gamma^{-1}(U^G)\to V\subset\mathbb{R}\times\dM,\\\nonumber
&(\varrho_G^{-1}\circ\gamma)(s)=(t(s),p(s))
\end{align}
for some continuous functions $t\colon\gamma^{-1}(U^G)\to\mathbb{R}$ and $p\colon\gamma^{-1}(U^G)\to\dM$.
We find
\begin{align}\nonumber
\varrho_g^{-1}\circ\tilde{\gamma}\colon&\gamma^{-1}(U^G)\to V\cap\{t\geq 0\},\\\nonumber
&(\varrho_g^{-1}\circ\tilde{\gamma}(s)=(|t(s)|,p(s))
\end{align}
for $\tilde{\gamma}$ in the geodesic collar, which is a consequence of equation~\eqref{reflectionequation}.

Now we give the actual proof.
For the sake of clarity, we will assume that $\dM$ is covered by a single chart without significant loss of generality.
Let $[a,b]\subset\gamma^{-1}(U^G)$ and $\varepsilon>0$.
Choose $\delta>0$ small enough so that the condition for absolute continuity of $\varrho_G^{-1}\circ\gamma|_{[a,b]}$ holds with $\varepsilon$ and $\delta$.

Let $m\in\mathbb{N}$ and let $\{(a_i,b_i)\}_{i=1}^m$ be a family of disjoint intervals with $[a_i,b_i]\subset[a,b]$ for all $1\leq i\leq m$ and total length $\sum_{i=1}^m|b_i-a_i|<\delta$.
Then we have
\begin{align}\nonumber
\sum_{i=1}^m&\Vert(\varrho_g^{-1}\circ\tilde{\gamma})(b_i)-(\varrho_g^{-1}\circ\tilde{\gamma})(a_i)\Vert\\\nonumber
&=\sum_{i=1}^m\Vert(|t(b_i)|,p(b_i))-(|t(a_i)|,p(a_i))\Vert\\\nonumber
&=\sum_{i=1}^m((|t(b_i)|-|t(a_i)|)^2+\Vert p(b_i)-p(a_i)\Vert^2)^{1/2}\\\nonumber
&\leq\sum_{i=1}^m(|t(b_i)-t(a_i)|^2+\Vert p(b_i)-p(a_i)\Vert^2)^{1/2}\\\nonumber
&=\sum_{i=1}^m\Vert(t(b_i),p(b_i))-(t(a_i),p(a_i))\Vert\\\nonumber
&=\sum_{i=1}^m\Vert(\varrho_G^{-1}\circ\gamma)(b_i)-(\varrho_G^{-1}\circ\gamma)(a_i)\Vert<\varepsilon
\end{align}
meaning that $\varrho_g^{-1}\circ\tilde{\gamma}|_{[a,b]}$ is absolutely continuous.
In total, $\tilde{\gamma}$ is absolutely continuous by Lemma~\ref{oneatlas}.

\textit{Regarding (b):}
Let $s_0\in I$ be a point where both $\gamma$ and $\tilde{\gamma}$ are differentiable.
According to Proposition~\ref{aedifferentiable} and part $(a)$, this is the case almost everywhere.

The claim $|\dot{\gamma}(s_0)|_G=|\dot{\tilde{\gamma}}(s_0)|_g$ is easy to prove if $\gamma(s_0)\notin\dM$ (cf. part $(a)$), so we consider $\gamma(s_0)\in\dM$.
Then it holds $t(s_0)=0$.
Since $\gamma$ and $\tilde{\gamma}$ are differentiable in $s_0$, so are $t,|t|$ and $p$.
By the generalised Gauss lemma, we have
\begin{alignat}{2}\nonumber
&\varrho_G^\ast G&&=\mathrm{d}t^2+G_t,\\\nonumber
&\varrho_g^\ast g&&=\mathrm{d}t^2+g_t
\end{alignat}
where $G_\bullet\colon V\to T^\ast\dM\otimes T^\ast\dM$ and $g_\bullet\colon V\cap\{t\geq 0\}\to T^\ast\dM\otimes T^\ast\dM$ coincide on $V\cap\{t\geq 0\}$.

It holds
\begin{align}\nonumber
&|\dot{\gamma}(s_0)|_G^2=\Bigl(\frac{\mathrm{d}}{\mathrm{d}s}\Bigr|_{s=s_0}t(s)\Bigr)^2+|\dot{p}(s_0)|_{g_0}^2,\\\nonumber
&|\dot{\tilde{\gamma}}(s_0)|_g^2=\Bigl(\frac{\mathrm{d}}{\mathrm{d}s}\Bigr|_{s=s_0}|t(s)|\Bigr)^2+|\dot{p}(s_0)|_{g_0}^2.
\end{align}
We show that $\mathrm{d}/\mathrm{d}s|_{s=s_0}t(s)=\mathrm{d}/\mathrm{d}s|_{s=s_0}|t(s)|=0$.

\textit{First case:}
There exists $\delta>0$ such that $t(s)\geq 0$ for all $s\in(s_0-\delta,s_0+\delta)$ or $t(s)\leq 0$ for all $s\in(s_0-\delta,s_0+\delta)$.
Assuming $t(s)\geq 0$ for all such $s$, the left derivative in $s_0$ is $\leq 0$ while the right derivative is $\geq 0$. 
Hence $\mathrm{d}/\mathrm{d}s|_{s=s_0}t(s)=0$.
Since $|t|$ coincides with $t$ on $(s_0-\delta,s_0+\delta)$, we also have $\mathrm{d}/\mathrm{d}s|_{s=s_0}|t(s)|=0$.

\textit{Second case:}
There is no such $\delta$ as in the first case.
We can pick a sequence $(s_k)_k$ with $s_k\to s_0$ so that $t(s_k)\geq 0$ for $k$ even and $t(s_k)\leq 0$ for $k$ odd.
For the even part, one has
\begin{equation}\nonumber
\frac{|t(s_{2k})|}{s_{2k}-s_0}=\frac{t(s_{2k})}{s_{2k}-s_0}\xrightarrow{k\to\infty}\frac{\mathrm{d}}{\mathrm{d}s}\Bigr|_{s=s_0}t(s)
\end{equation}
and for the odd part
\begin{equation}\nonumber
\frac{|t(s_{2k+1})|}{s_{2k+1}-s_0}=\frac{-t(s_{2k+1})}{s_{2k+1}-s_0}\xrightarrow{k\to\infty}-\frac{\mathrm{d}}{\mathrm{d}s}\Bigr|_{s=s_0}t(s).
\end{equation}
By differentiability of $|t|$, we conclude $\mathrm{d}/\mathrm{d}s|_{s=s_0}t(s)=-\mathrm{d}/\mathrm{d}s|_{s=s_0}t(s)$ or $\mathrm{d}/\mathrm{d}s|_{s=s_0}t(s)=0$.
Therefore,
\begin{equation}\nonumber
\frac{\mathrm{d}}{\mathrm{d}s}\Bigr|_{s=s_0}|t(s)|=\frac{\mathrm{d}}{\mathrm{d}s}\Bigr|_{s=s_0}t(s)=-\frac{\mathrm{d}}{\mathrm{d}s}\Bigr|_{s=s_0}t(s)=0.
\end{equation}
In either case, it holds $|\dot{\gamma}(s_0)|_G=|\dot{p}(s_0)|_{g_0}=|\dot{\tilde{\gamma}}(s_0)|_g$.
\end{proof}
\begin{remark}
The doubling property for $g$ is actually not needed in part $(a)$, since gluing together two copies of an appropriate collar neighbourhood of $\dM\subset M$ yields a tubular neighbourhood for $\dM$ in $\mathsf{D}M$.
This is by construction of the smooth structure on $\mathsf{D}M$.
\end{remark}

Now Proposition~\ref{samedistance} follows immediately from Corollary~\ref{dacdg} and Proposition~\ref{Reflectionargument}.
\begin{proof}[Proof of Theorem \ref{completemetrics}]
Let $(p_k)_k$ be a Cauchy sequence in $\mathsf{D}M$ with respect to $d_G$.
Without loss of generality, we can assume that all points $p_k,k\in\mathbb{N}$ are contained in a single copy of $M$.
If not, we select an appropriate subsequence.
Then $(p_k)_k$ converges by Proposition~\ref{samedistance} and the completeness of $g$.
\end{proof}
\section{The deformation principle}
Throughout this section we fix a smooth manifold $M$ of dimension $n\geq 2$ with non-empty (possibly non-compact) boundary.
For each Riemannian metric $g$ on $M$ we denote by
\begin{enumerate}
\item[\myicon]{$\scal_g\colon M\to\mathbb{R}$ its scalar curvature;}
\item[\myicon]{$g_0\in C^{\infty}(\partial M;T^\ast\partial M\otimes T^\ast\partial M)$ the metric induced on $\partial M$;}
\item[\myicon]{$N_g$ the inward pointing unit normal vector field along $\partial M$;}
\item[\myicon]{$\mathrm{II}_g$ the second fundamental form of $\partial M\subset M$ with respect to $N_g$;}
\item[\myicon]{$H_g=\frac{1}{n-1}\tr_{g_0}(\mathrm{II}_g)\colon\partial M\to\mathbb{R}$ the mean curvature of $\partial M$.}
\end{enumerate}
The space of Riemannian metrics on $M$ is denoted by $\mathscr{R}(M)$.
It is endowed with the weak $C^{\infty}$-topology.
Let $\sigma\colon M\to\mathbb{R}$ be a continuous function that also remains unchanged.
We investigate metrics in the subspace $\mathscr{R}_{>\sigma}(M):=\{g\in\mathscr{R}(M):\scal_g>\sigma\}$, where the topology is inherited from $\mathscr{R}(M)$.

The following refined versions of geodesic collar neighbourhoods are important for the deformation principle:
Given a metric $g$ on $M$ and a neighbourhood $\mathscr{U}\subset M$ of $\partial M$, there exists a continuous positive function $\eta\colon\partial M\to\mathbb{R}$ such that $U^g=U^g_\eta$ is the diffeomorphic image of
\begin{equation}\nonumber
V_\eta:=\{(t,p)\in[0,\infty)\times\partial M:t<\eta(p)\}
\end{equation}
under the normal exponential map, i.e.
\begin{equation}\nonumber
\varrho_g\colon V_\eta\to U^g_\eta\,,\,\varrho_g(t,p)=\exp_p(tN_g(p))
\end{equation}
is a geodesic collar neighbourhood.
We can arrange $U^g_\eta\subset\mathscr{U}$ by choosing $\eta$ small enough.
Working in families of metrics is more complicated because the neighbourhoods $V_\eta$ and $U^g_\eta$ may differ between members.
However, it is possible to fix either the domain or the target in the following sense:
\begin{factC1}\label{C1}
Let $K$ be a compact Hausdorff space and let $g\colon K\to\mathscr{R}(M)$ be a continuous family of Riemannian metrics.
For each neighbourhood $\mathscr{U}$ of $\partial M$ there exists a continuous positive function $\eta\colon\partial M\to\mathbb{R}$ such that
\begin{enumerate}
\item[\myicon]{$\varrho_{g(\xi)}\colon V_\eta\to U^{g(\xi)}_\eta$ is a collar neighbourhood for each $\xi\in K$;}
\item[\myicon]{$\varrho\colon K\to C^{\infty}(V_\eta;M),\xi\mapsto\varrho_{g(\xi)}$ is continuous;}
\item[\myicon]{$U^{g(\xi)}_\eta\subset\mathscr{U}$ for all $\xi\in K$.}
\end{enumerate}
\end{factC1}
\begin{factC2}
Let $K$ be a compact Hausdorff space and let $g\colon K\to\mathscr{R}(M)$ be a continuous family of Riemannian metrics.
Let $\eta\colon\partial M\to\mathbb{R}$ be a continuous positive function such that $\varrho_{g(\xi)}\colon V_\eta\to U^{g(\xi)}_\eta$ is a geodesic collar neighbourhood for all $\xi\in K$.
Then there exists a neighbourhood $U\subset M$ of $\partial M$ with the following properties:
\begin{enumerate}
\item[\myicon]{$U\subset U^{g(\xi)}_\eta$ for all $\xi\in K$;}
\item[\myicon]{$\varrho^{-1}\colon K\to C^{\infty}(U;V_\eta),\xi\mapsto\varrho_{g(\xi)}^{-1}$ is continuous.}
\end{enumerate}
\end{factC2}
As with $\mathscr{R}(M)$, we equip all function spaces with the weak $C^\infty$-topology.
Facts CN1 and CN2 are proven in \cite[Satz 2.14 and Prop.~2.15]{Frerichs}.

Let $I\subset\mathbb{R}$ be an interval with $[0,1]\subset I$.
A continuous family of \textit{time-dependent Riemannian metrics} is a continuous map $g\colon K\to C^{\infty}(M\times I;T^\ast M\otimes T^\ast M)$ with $g(\xi)(\bullet,s)$ beeing a Riemannian metric for each $\xi\in K$ and $s\in I$.
We require $I$ to be open as well as $g(\xi)(p,s)=g(\xi)(p,0)$ for all $s\leq 0$ and $g(\xi)(p,s)=g(\xi)(p,1)$ for all $s\geq 1$.
This condition is not inalienable, but for avoiding corners.

There is a need for adapted versions of CN1 and CN2 with families of time-dependent metrics.
As an abbreviation, we will denote $\varrho_{\xi,s}:=\varrho_{g(\xi)(\bullet,s)}$ and $U_\eta^{\xi,s}:=U_\eta^{g(\xi)(\bullet,s)}$.
\begin{factC1'}
Let $K$ be a compact Hausdorff space and let $g\colon K\to C^{\infty}(M\times I;T^\ast M\otimes T^\ast M)$ be a continuous family of time-dependent Riemannian metrics.
For each neighbourhood $\mathscr{U}$ of $\partial M$ there exists a continuous positive function $\eta\colon\partial M\to\mathbb{R}$ such that
\begin{enumerate}
\item[\myicon]{$\varrho_{\xi,s}\colon V_\eta\to U_\eta^{\xi,s}$ is a collar neighbourhood for each $\xi\in K$ and $s\in I$;}
\item[\myicon]{$\varrho\colon K\to C^{\infty}(V_\eta\times I;M),\xi\mapsto[(t,p,s)\mapsto\varrho_{\xi,s}(t,p)]$ is well-defined and continuous;}
\item[\myicon]{$U_\eta^{\xi,s}\subset\mathscr{U}$ for all $\xi\in K$ and $s\in I$.}
\end{enumerate}
\end{factC1'}
\begin{factC2'}
Let $K$ be a compact Hausdorff space and let $g\colon K\to C^{\infty}(M\times I;T^\ast M\otimes T^\ast M)$ be a continuous family of time-dependent Riemannian metrics.
Let $\eta\colon\partial M\to\mathbb{R}$ be a continuous positive function such that $\varrho_{\xi,s}\colon V_\eta\to U_\eta^{\xi,s}$ is a geodesic collar neighbourhood for all $\xi\in K$ and $s\in I$.
Then there exists a neighbourhood $U\subset M$ of $\partial M$ with the following properties:
\begin{enumerate}
\item[\myicon]{$U\subset U^{\xi,s}$ for all $\xi\in K$ and $s\in I$;}
\item[\myicon]{$\varrho^{-1}\colon K\to C^{\infty}(U\times I;V_\eta),\xi\mapsto[(p,s)\mapsto\varrho_{\xi,s}^{-1}(p)]$ is well-defined and continuous.}
\end{enumerate}
\end{factC2'}
To say $\varrho$ is well-defined means that the assignment $(t,p,s)\mapsto \varrho_{\xi,s}(t,p)$ is jointly smooth with respect to the point $(t,p)\in V_\eta$ and the parameter $s\in I$ (similar with $\varrho^{-1}$).
Facts CN1' and CN2' can be deduced from \cite[Satz 2.14 and Prop.~2.15]{Frerichs} using autonomization techniques from ODE theory.

\subsection{Uniformisation for geodesic collar neighbourhoods}
To begin the deformation process, it is useful to standardise the normal exponential maps of a given continuous family $g\colon K\to\mathscr{R}(M)$.
Uniformisation is even necessary for some applications, see Corollary~\ref{fromHKRS08}.
The process concludes with Corollary~\ref{uniformization}.
\begin{lemma}\label{cover1}
Let $r>0$.
There exists a complete Riemannian metric $m\in\mathscr{R}(M)$ and an open cover $(U_i)_{i\in I}$ such that each $U_i$ is relatively compact and that
\begin{equation}\nonumber
I_i:=\{j\in I:\overline{B}_m(U_i,r)\cap\overline{B}_m(U_j,r)\neq\varnothing\}
\end{equation}
is finite for every $i\in I$.
Here $\overline{B}_m(U_i,r)=\{p\in M:d_m(p,U_i)\leq r\}$ denotes the closed $r$-neighbourhood around $U_i$ with respect to $m$.
\end{lemma}
\begin{proof}
We choose a proper smooth embedding $\Phi\colon M\to\mathbb{R}^N$ of $M$ into some sufficiently large Euclidean space $\mathbb{R}^N$.
The sets
\begin{equation}\nonumber
U_i:=\Phi^{-1}(B(i,\sqrt{N})),\;i\in I:=\mathbb{Z}^N
\end{equation}
form an open cover of $M$, where $B(i,\sqrt{N})$ denotes the open Euclidean ball of radius $\sqrt{N}$ around $i\in I$.
Each fixed $U_i$ intersects only a finite number of other $U_j,j\in I$.
The sets are relatively compact, since every closure $\cl(U_i)\subset\Phi^{-1}(\overline{B}(i,\sqrt{N}))$ is contained in a compact neighbourhood.

We set
\begin{equation}\nonumber
m:=\Phi^\ast g_{\mathrm{eucl}}
\end{equation}
where $g_{\mathrm{eucl}}$ denotes the Euclidean metric in $\mathbb{R}^N$.
The metric $m$ is complete because $\Phi(M)\subset\mathbb{R}^N$ is closed.
It holds $\Phi(\overline{B}_m(U_i,r))\subset\overline{B}(i,\sqrt{N}+r)$ for every $i\in I$.
Since the balls $\overline{B}(i,\sqrt{N}+r)$ meet only in finite families, the same is true for the neighbourhoods $\overline{B}_m(U_i,r)$.
\end{proof}

We choose $r=1$ and fix a background metric $m\in\mathscr{R}(M)$ together with an open cover $(U_i)_{i\in I}$ of $M$ as in the lemma.
We would like to mention that the balls $\overline{B}_m(U_i,1)$ are compact due to the Hopf-Rinow theorem, see \cite[Prop.~2.5.22 and Thm.~2.5.28]{BBI}.
This is applicable since Riemannian manifolds with boundary are length spaces, see e.g. \cite[Prop.~3.18]{BrH}.

Suppose there is a distinguished metric $g(\xi_0),\xi_0\in K$ in a given family $g\colon K\to\mathscr{R}(M)$.
The starting point for the uniformisation process is to join every metric $g(\xi),\xi\in K$ with the metric $g(\xi_0)$ by a path of metrics.

Let $I\subset\mathbb{R}$ be an open interval with $[0,1]\subset I$ and let $\rho\colon\mathbb{R}\to[0,1]$ be a smooth function with
\begin{enumerate}
\item[\myicon]{$\rho(s)=0$ for all $s\leq 0$;}
\item[\myicon]{$\rho(s)=1$ for all $s\geq 1$;}
\item[\myicon]{$\supp\rho'\subset(0,1)$.}
\end{enumerate}
Let $\mathscr{U}$ be a neighbourhood of $\dM$. We consider the family of time-dependent Riemannian metrics given by $(1-\rho(s))g(\xi_0)+\rho(s)g(\xi)$ for $\xi\in K$ and $s\in I$.

By Fact CN1', there exists an associated family of geodesic collar neighbourhoods $\varrho_{\xi,s}\colon V\to U^{\xi,s}$ that satisfy the inclusion $U^{\xi,s}\subset\mathscr{U}$ for all $\xi\in K$ and $s\in I$.
We drop the subscript $\eta$ at this point.

Put
\begin{align}\nonumber
\omega\colon&K\to C^{\infty}(U^{g(\xi_0)}\times I;M),\\\nonumber
&\omega(\xi)(p,s)=(\varrho_{\xi,s}\circ\varrho_{g(\xi_0)}^{-1})(p).
\end{align}
This map is continuous due to Fact CN1'.

Our proofs will utilise time-dependent vector fields.
Similar to time-dependent metrics, a \textit{time-dependent vector field} $V\in C^{\infty}(M\times I;TM)$ is supposed to satisfy $V(\xi)(p,s)=V(\xi)(p,0)$ for all $s\leq 0$ and $V(\xi)(p,s)=V(\xi)(p,1)$ for all $s\geq 1$.
We further require $V(\dM\times I)\subset T(\dM)$.

The vector field is said to have \emph{bounded velocity} if there exists a complete Riemannian metric $g\in\mathscr{R}(M)$ and a constant $L\geq 0$ such that $|V(p,s)|_g\leq L$ for all $p\in M$ and $s\in I$, cf. \cite[Ch.~8]{Hirsch}.
\begin{lemma}\label{vector field}
Let $(K,\xi_0)$ be a compact pointed Hausdorff space and let $g\colon K\to\mathscr{R}(M)$ be a continuous family of Riemannian metrics.
Let $\delta\colon M\to\mathbb{R}$ be a continuous positive function.

For each neighbourhood $\mathscr{U}$ of $\dM$, there exists a smaller neighbourhood $\dM\subset U\subset\mathscr{U}$ and a continuous family
\begin{equation}\nonumber
V\colon K\to C^{\infty}(M\times I;TM)
\end{equation}
of time-dependent vector fields such that the following holds for all $\xi\in K$ and $s\in I$:
\begin{enumerate}
\item[(a)]{$V(\xi)(p,s)=0$ for all $p\in\partial M$;}
\item[(b)]{$V(\xi_0)(p,s)=0$ for all $p\in M$;}
\item[(c)]{The curve $\omega(\xi)(p,\bullet)\colon I\to M$ is a solution to the IVP
\begin{align}\nonumber
&\dot{\gamma}(s)=V(\xi)(\gamma(s),s)\\\nonumber
&\gamma(0)=p
\end{align}
for all $p\in U$;}
\item[(d)]{$V(\xi)(p,s)=0$ for all $p\in M\setminus\cl(\mathscr{U})$;}
\item[(e)]{$|V(\xi)(p,s)|_m<\delta(p)$ for all $p\in M$;}
\item[(f)]{$|V(\xi)(p,s)|_{g(\xi)}<\delta(p)$ for all $p\in M$.}
\end{enumerate}
\end{lemma}
\begin{proof}
Let $\mathscr{U}$ be a neighbourhood of $\dM$.
We will start from the point before the lemma:
There exist neighbourhoods $\dM\subset U''\subset U'\subset M$ with $\cl(U'')\subset U'\subset U^{\xi,s}$ for all $\xi\in K$ and $s\in I$, as can be seen from Fact CN1',CN2' and the later Lemma~\ref{lemma3.6}.
Then
\begin{align}\nonumber
v\colon&K\to C^{\infty}(U'\times I;TM),\\\nonumber
&v(\xi)(p,s)=\frac{\partial\omega(\xi)}{\partial s}((\varrho_{g(\xi_0)}\circ\varrho_{\xi,s}^{-1})(p),s)
\end{align}
is a local family of time-dependent vector fields.
For each point $p\in\partial M$, it holds $\omega(\xi)(p,s)=p$ for all $\xi\in K$ and $s\in I$ so that $v\equiv 0$ on $\partial M$.
By continuity and compactness of $K$, we can assume that
\begin{alignat}{2}\nonumber
&|v(\xi)(p,s)|_m&&<\delta(p),\\\nonumber
&|v(\xi)(p,s)|_{g(\xi)}&&<\delta(p)
\end{alignat}
for all $p\in U',\xi\in K$ and $s\in I$.
This is without loss of generality.

Next, let $V''\subset[0,\infty)\times\dM$ be a neighbourhood of $\dM$ such that $\varrho_{\xi,s}(V'')\subset U''$ for all $\xi\in K$ and $s\in I$.
Put $U:=\varrho_{g(\xi_0)}(V'')$.
We then obtain the differential equation
\begin{equation}\nonumber
\frac{\partial\omega(\xi)}{\partial s}(p,s)=v(\xi)(\omega(\xi)(p,s),s)
\end{equation}
for all $p\in U,\xi\in K$ and $s\in I$.
In order to globalise the construction, let $\chi\colon M\to[0,1]$ be a smooth bumping function with
\begin{enumerate}
\item[\myicon]{$\chi=1$ on $U''$;}
\item[\myicon]{$\supp\chi\subset U'$.}
\end{enumerate}
Define a global time-dependent vector field by
\begin{align}\nonumber
V\colon&K\to C^{\infty}(M\times I;TM),\\\nonumber
&V(\xi)(p,s)=\chi(p)\cdot v(\xi)(p,s).
\end{align}
It can be verified that $V$ satisfies properties $(a)-(f)$.
We present $(c)$:

It holds $\omega(\xi)(p,s)\in U''$ for all $p\in U,\xi\in K$ and $s\in I$ by construction.
We further know that $\omega$ satisfies the IVP
\begin{align}\nonumber
\frac{\partial\omega(\xi)}{\partial s}(p,s)&=v(\xi)(\omega(\xi)(p,s),s),\\\nonumber
\omega(\xi)(p,0)&=p.
\end{align}
Since $V$ and $v$ coincide on $U''$, the curves $\omega(\xi)(p,\bullet)\colon I\to M$ are solutions to $(c)$.
\end{proof}
\begin{proposition}\label{diff}
Let $(K,\xi_0)$ be a compact pointed Hausdorff space and let $g\colon K\to\mathscr{R}(M)$ be a continuous family of Riemannian metrics.
Let $\delta\colon M\to\mathbb{R}$ be a continuous positive function.

For each neighbourhood $\mathscr{U}$ of $\dM$, there exists a smaller neighbourhood $\dM\subset U\subset\mathscr{U}$ and a continuous map
\begin{equation}\nonumber
\Omega\colon K\times[0,1]\to\Diff(M)
\end{equation}
such that the following holds for all $\xi\in K$ and $s\in I$:
\begin{enumerate}
\item[(a)]{$\Omega(\xi,0)=\id_M$;}
\item[(b)]{$\Omega(\xi_0,s)=\id_M$;}
\item[(c)]{$\Omega(\xi,1)=\varrho_{g(\xi)}\circ\varrho_{g(\xi_0)}^{-1}$ on U;}
\item[(d)]{$\Omega(\xi,s)=\id$\quad and\quad$\mathrm{d}\Omega(\xi,s)=\id$ on $M\setminus\mathscr{U}$;}
\item[(e)]{$d_m(\Omega(\xi,s)(p),p)<\delta(p)$ for all $p\in M$;}
\item[(f)]{$d_{g(\xi)}(\Omega(\xi,s)(p),p)<\delta(p)$ for all $p\in M$.}
\end{enumerate}
\end{proposition}
\begin{proof}
Let $\mathscr{U}$ be a neighbourhood of $\dM$.
We apply Lemma~\ref{vector field} for the function $\delta/2$ to obtain a family of vector fields $V$ and a neighbourhood $\partial M\subset U_V\subset\mathscr{U}$ with properties~\ref{vector field}$(a)-(f)$.
Now we use the open cover $(U_i)_{i\in I}$ from Lemma~\ref{cover1}:

Let $i\in I$.
By compactness, there exists some $\tilde{\delta}_i>0$ such that
\begin{equation}\label{Diff:eq1}
||V(\xi)(p,s)|_m-|V(\xi)(q,s)|_m|<\frac{1}{2}\min\{\delta(p'):p'\in\overline{B}_m(U_i,1)\}
\end{equation}
and
\begin{equation}\label{Diff:eq2}
||V(\xi)(p,s)|_{g(\xi)}-|V(\xi)(q,s)|_{g(\xi)}|<\frac{1}{2}\min\{\delta(p'):p'\in\overline{B}_m(U_i,1)\}
\end{equation}
for all $p,q\in\overline{B}_m(U_i,1)$ with $d_m(p,q)<\tilde{\delta}_i$ and for all $\xi\in K,s\in I$.
Let $(\psi_i)_{i\in I}$ be a partition of unity subordinate to $(B_m(U_i,1))_{i\in I}$.
Set
\begin{equation}\nonumber
\tilde{\delta}\colon M\to\mathbb{R}\,,\,\tilde{\delta}(p)=\sum_{i\in I}\psi_i(p)\cdot\min_{j\in I_i}\tilde{\delta}_j.
\end{equation}
This is a well-defined continuous function with $\tilde{\delta}(p)>0$ for all $p\in M$.
We have $\delta(p)\leq\tilde{\delta}_i$ for all $i\in I$ and $p\in\overline{B}_m(U_i,1)$.

Since $V\equiv 0$ on $\partial M$, there is a neighbourhood $\partial M\subset U'\subset M$ such that
\begin{equation}\label{Diff:eq3}
\begin{alignedat}{2}
&|V(\xi)(p,s)|_m&&<\min\{\tilde{\delta}(p),1\},\\
&|V(\xi)(p,s)|_{g(\xi)}&&<\tilde{\delta}(p)
\end{alignedat}
\end{equation}
for all $p\in U',\xi\in K$ and $s\in I$.
Choose a smooth bumping function $\chi\colon M\to[0,1]$ with
\begin{enumerate}
\item[\myicon]{$\chi=1$ on some neighbourhood $U''$ of $\dM$;}
\item[\myicon]{$\supp\chi\subset U'$.}
\end{enumerate}
Then set
\begin{align}\nonumber
V_{\Omega}\colon&K\to C^{\infty}(M\times I;TM),\\\nonumber
&V_\Omega(\xi)(p,s)=\chi(p)\cdot V(\xi)(p,s).
\end{align}
This family satisfies the inequalities~\eqref{Diff:eq3} globally and it satisfies $|V_\Omega(\xi)|\leq|V(\xi)|$ for every Riemannian metric.
In particular, $V_\Omega(\xi)$ has bounded velocity for every $\xi\in K$.

According to \cite[Thm.~8.1.1]{Hirsch} and \cite[Satz A.5]{Frerichs}, the solution curves to the IVP
\begin{align}\nonumber
&\dot{\gamma}(s)=V_\Omega(\xi)(\gamma(s),s)\\\nonumber
&\gamma(0)=p,\;\;p\in M
\end{align}
assemble in a well-defined and continuous solution map
\begin{equation}\nonumber
\Omega\colon K\to C^{\infty}(M\times I;M)
\end{equation}
which consists of diffeomorphisms.
From this we deduce a map
\begin{equation}\nonumber
\Omega\colon K\times[0,1]\to\Diff(M)
\end{equation}
which is again denoted by $\Omega$, and check properties $(a)-(f)$:
\begin{enumerate}
\item[$(a)$]{is by definition of flows.}
\item[$(b)$]{follows from Lemma~\ref{vector field}$(b)$.}
\item[$(c)$]{As in the proof of Lemma~\ref{vector field}, we find a neighbourhood $U\subset U_V\cap U''$ such that
\begin{align}\nonumber
\frac{\partial\omega(\xi)}{\partial s}(p,s)&=V_{\Omega}(\omega(\xi)(p,s),s),\\\nonumber
\omega(\xi)(p,0)&=p
\end{align}
for all $p\in U,\xi\in K$ and $s\in I$.
By uniqueness of integral curves, it holds $\Omega(\xi)(p,s)=\omega(\xi)(p,s)$ for $p\in U$.
In particular, $\Omega(\xi,1)=\varrho_{g(\xi)}\circ\varrho^{-1}_{g(\xi_0)}$ on $U$ for all $\xi\in K$.}
\item[$(d)$]{follows from the fact that $V_\Omega\equiv 0$ on $M\setminus\cl(\mathscr{U})$.}
\item[$(e)$]{Let $p\in M$.
Then it holds $p\in U_i$ for some $i\in I$.
For any $\xi\in K$ and $s\in[0,1]$, we have
\begin{align}\nonumber
d_m(\Omega(\xi,s)(p),p)&=d_m(\Omega(\xi,s)(p),\Omega(\xi,0)(p))\\\nonumber
&\leq\int_0^s\left|\frac{\partial\Omega(\xi)}{\partial t}(p,t)\right|_m\;\mathrm{d}t\\\nonumber
&=\int_0^s|V_\Omega(\xi)(\Omega(\xi,t)(p),t)|_m\;\mathrm{d}t\\\nonumber
&\leq\int_0^s1\;\mathrm{dt}\leq 1.
\end{align}
This means $\Omega(\xi,s)(p)\in\overline{B}_m(U_i,1)$ for all $\xi\in K$ and $s\in[0,1]$.
Therefore, it holds
\begin{align}\nonumber
d_m(\Omega(\xi,s)(p),p)&\leq\int_0^s|V_\Omega(\xi)(\Omega(\xi,t)(p),t)|_m\;\mathrm{d}t\\\nonumber
&<\int_0^s\tilde{\delta}(\Omega(\xi,t)(p))\;\mathrm{d}t\\\nonumber
&\leq\int_0^s\tilde{\delta}_i\;\mathrm{d}t\leq\tilde{\delta}_i
\end{align}
for all $\xi\in K$ and $s\in[0,1]$.
Finally, we obtain
\begin{align}\nonumber
d_m(\Omega(\xi,s)(p),p)&\leq\int_0^s|V_\Omega(\xi)(\Omega(\xi,t)(p),t)|_m\;\mathrm{d}t\\\nonumber
&\leq\int_0^s|V(\xi)(\Omega(\xi,t)(p),t)|_m\;\mathrm{d}t\\\nonumber
&\leq\int_0^s|V(\xi)(p,t)|_m+||V(\xi)(p,t)|_m-|V(\xi)(\Omega(\xi,t)(p),t)|_m|\;\mathrm{d}t\\\nonumber
&<\int_0^s\frac{1}{2}\delta(p)+\frac{1}{2}\min\{\delta(p'):p'\in\overline{B}_m(U_i,1)\}\;\mathrm{d}t\\\nonumber
&\leq\delta(p)
\end{align}
for all $\xi\in K$ and $s\in[0,1]$ by inequality~\eqref{Diff:eq1}.}
\item[$(f)$]{Let $p\in M$.
Choose some $i\in I$ with $p\in U_i$.
We know that $\Omega(\xi,s)(p)\in\overline{B}_m(U_i,1)$ and that $d_m(\Omega(\xi,s)(p),p)<\tilde{\delta}_i$ for all $\xi\in K$ and $s\in[0,1]$.
The assertion follows from inequality ~\eqref{Diff:eq2}.\qedhere}
\end{enumerate}
\end{proof}

\begin{remark}
Hirsch \cite[Thm.~8.1.1]{Hirsch} refers to time-dependent vector fields that are parametrized over compact intervals.
The result is also true for open intervals.
\end{remark}

\begin{corollary}\label{uniformization}
Let $K$ be a compact Hausdorff space and let $g\colon K\to\mathscr{R}_{>\sigma}(M)$ be a continuous family of Riemannian metrics of scalar curvature greater than $\sigma$.
Let $\beta\in\mathscr{R}(M)$ be another distinguished Riemannian metric.

For each neighbourhood $\mathscr{U}$ of $\dM$, there exists a smaller neighbourhood $\dM\subset U\subset\mathscr{U}$ and a continuous map
\begin{equation}\nonumber
f\colon K\times[0,1]\to\mathscr{R}_{>\sigma}(M)
\end{equation}
so that the following holds for all $\xi\in K$ and $s\in[0,1]$:
\begin{enumerate}
\item[(a)]{$f(\xi,0)=g(\xi)$;}
\item[(b)]{$N_{f(\xi,1)}=N_{\beta}$ and $\varrho_{f(\xi,1)}=\varrho_{\beta}$ on $\varrho_{\beta}^{-1}(U)$;}
\item[(c)]{$f(\xi,s)_t=g(\xi)_t$ near $\{0\}\times\dM$;}
\item[(d)]{$f(\xi,s)=g(\xi)$ on $M\setminus\mathscr{U}$;}
\item[(e)]{$f(\xi,s)$ is quasi-isometric to $g(\xi)$ via the identity;}
\item[(f)]{if $g(\xi)$ is complete, then so is $f(\xi,s)$.}
\end{enumerate}
\end{corollary}
Property $(e)$ means that there exist constants $A\geq 1$ and $B\geq 0$ such that
\begin{equation}\nonumber
\frac{1}{A}\,d_{g(\xi)}(p,q)-B\leq d_{f(\xi,s)}(p,q)\leq A\,d_{g(\xi)}(p,q)+B
\end{equation}
for all $p,q\in M$.
In fact, we can arrange these inequalities for $A=1$ and an arbitrary constant $B>0$ independent from $\xi$ and $s$.
The proof below works with $B=1$.
\begin{proof}
Let $\mathscr{U}$ be a neighbourhood of $\dM$.
Let $\xi_0$ be an arbitrary object.
Then $K\cup\{\xi_0\}$ with the topology of disjoint union is again a compact Hausdorff space.
We put $g(\xi_0):=\beta$.

Each metric $f(\xi,s)$ will be a pullback of $g(\xi)$ along some diffeomorphism of $M$.
For the construction, we consider the open cover $(U_i)_{i\in I}$ from Lemma~\ref{cover1}.
Another important quantity is the scalar curvature surplus
\begin{equation}\nonumber
D\colon K\to C^{\infty}(M;\mathbb{R})\,,\,D(\xi)(p)=\scal_{g(\xi)}(p)-\sigma(p).
\end{equation}
Now let $i\in I$.
By compactness, there exists some $\delta_i>0$ such that
\begin{equation}\nonumber
|\scal_{g(\xi)}(p)-\scal_{g(\xi)}(q)|<\min\{D(\xi')(p'):p'\in\overline{B}_m(U_i,1),\xi'\in K\}
\end{equation}
for all $p,q\in\overline{B}_m(U_i,1)$ with $d_m(p,q)<\delta_i$ and for all $\xi\in K$.

Let $(\psi_i)_{i\in I}$ be a partition of unity subordinate to $(U_i)_{i\in I}$. Set
\begin{equation}\nonumber
\delta\colon M\to\mathbb{R}\,,\,\delta(p)=\frac{1}{2}\sum_{i\in I}\psi_i(p)\cdot\min\bigl\{\min_{j\in I_i}\delta_j,1\bigr\}.
\end{equation}
This is a well-defined continuous function with $0<\delta(p)\leq\frac{1}{2}$ for all $p\in M$.
Furthermore, it holds $\delta(p)<\delta_i$ for all $i\in I$ and $p\in U_i$.

Proposition~\ref{diff}, applied for the compact Hausdorff space $K\cup\{\xi_0\}$ and the function $\delta$, yields a neighbourhood $\dM\subset U\subset\mathscr{U}$ and a family $\Omega\colon(K\cup\{\xi_0\})\times[0,1]\to\Diff(M)$ with its distinguished properties~\ref{diff}$(a)-(f)$.
We put
\begin{equation}\nonumber
f\colon K\times[0,1]\to\mathscr{R}(M)\,,\,f(\xi,s):=\Omega(\xi,s)^\ast g(\xi).
\end{equation}
This operation preserves the lower scalar curvature bound $\sigma$, which is shown as follows:
Let $p\in M,\xi\in K$ and $s\in[0,1]$.
Then it holds $\scal_{f(\xi,s)}(p)=\scal_{g(\xi)}(\Omega(\xi,s)(p))$.

Choose some $i\in I$ with $p\in U_i$.
By Proposition~\ref{diff}$(e)$, we have $d_m(\Omega(\xi,s)(p),p)<\delta(p)<1$, i.e. $\Omega(\xi,s)(p)\in\overline{B}_m(U_i,1)$.
The same inequality gives $d_m(\Omega(\xi,s)(p),p)<\delta(p)<\delta_i$ so that
\begin{align}\nonumber
&|\scal_{g(\xi)}(\Omega(\xi,s)(p))-\scal_{g(\xi)}(p)|\\\nonumber
<\;&\min\{D(\xi')(p'):p'\in\overline{B}_m(U_i,1),\xi'\in K\}\leq D(\xi)(p).
\end{align}
Hence, we obtain $\scal_{f(\xi,s)}(p)>\sigma(p)$.

It remains to check properties $(a)-(f)$:
\begin{enumerate}
\item[$(a)$]{follows from Proposition~\ref{diff}$(a)$;}
\item[$(b)$]{Let $\xi\in K$ and $(t,p)\in\varrho_{\beta}^{-1}(U)$.
The map $\Omega(\xi,1)\colon(M,f(\xi,1))\to(M,g(\xi))$, as an isometry, maps normal geodesics to normal geodesics.
By Proposition~\ref{diff}$(c)$ we have
\begin{align}\nonumber
\Omega(\xi,1)(\varrho_{\beta}(t,p))=\varrho_{g(\xi)}(t,p)&=\exp_p^{g(\xi)}(tN_{g(\xi)}(p))\\\nonumber
&=\Omega(\xi,1)(\exp_p^{f(\xi,1)}(tN_{f(\xi,1)}(p)))\\\nonumber
&=\Omega(\xi,1)(\varrho_{f(\xi,1)}(t,p))
\end{align}
since $\Omega$ preserves boundary points.
Thus, $\varrho_{f(\xi,1)}(t,p)=\varrho_{\beta}(t,p)$.
In particular,
\begin{equation}\nonumber
N_{f(\xi,1)}(p)=\left.\frac{\partial}{\partial t}\right|_{t=0}\varrho_{f(\xi,1)}(t,p)=\left.\frac{\partial}{\partial t}\right|_{t=0}\varrho_{\beta}(t,p)=N_{\beta}(p).
\end{equation}}
\item[$(c)$]{Let $\xi\in K$.
We choose a neighbourhood $V\subset[0,\infty)\times\partial M$ of $\dM$ such that $\varrho_{f(\xi,s)}\colon V\to U^{f(\xi,s)}$ is a geodesic collar neighbourhood with $U^{f(\xi,s)}\subset U$ for all $s\in[0,1]$.
Consider the diagram
\begin{equation}\nonumber
\begin{tikzcd}
(V,\mathrm{d}t^2+g(\xi)_t)\arrow{rr}{\varrho_{g(\xi)}}&&(U^{g(\xi)},g(\xi))\\
\\
(V,\mathrm{d}t^2+f(\xi,s)_t)\arrow{rr}{\varrho_{f(\xi,s)}}\arrow{uu}{\id}&&(U^{f(\xi,s)},f(\xi,s))\arrow{uu}{\Omega(\xi,s)}
\end{tikzcd}
\end{equation}
It commutes since $\Omega(\xi,s)$ maps normal geodesics to normal geodesics.
As the horizontal maps and the vertical map on the right are isometries, so is the identity on the left.}
\item[$(d)$]{follows from Proposition~\ref{diff}$(d)$.}
\item[$(e)$]{Let $p,q\in M,\xi\in K$ and $s\in[0,1]$. According to Proposition~\ref{diff}$(f)$, it holds
\begin{align}\nonumber
d_{f(\xi,s)}(p,q)&=d_{g(\xi)}(\Omega(\xi,s)(p),\Omega(\xi,s)(q))\\\nonumber
&\leq d_{g(\xi)}(p,q)+d_{g(\xi)}(\Omega(\xi,s)(p),p)+d_{g(\xi)}(\Omega(\xi,s)(q),q)\\\nonumber
&\leq d_{g(\xi)}(p,q)+1.
\end{align}
The other inequality is similar.}
\item[$(f)$]{Let $\xi\in K$ and $s\in[0,1]$.
Since $\Omega(\xi,s)\colon(M,f(\xi,s))\to(M,g(\xi))$ is a Riemannian isometry, it is also a metric isometry.
Completeness is transferred from $g(\xi)$ to $f(\xi,s)$.\qedhere}
\end{enumerate}
\end{proof}
\subsection{A two-stage deformation scheme}
In the next step, we establish a standard type of Riemannian metrics within the deformation principle.
These are called \emph{C-normal metrics}.
They give better control over the scalar curvature.

For the definition, we consider a Riemannian metric $g$ on $M$ with its representation $g=\mathrm{d}t^2+g_t$ in a geodesic collar neighbourhood $\varrho_g\colon V\to U^g$.
Then it holds $\mathrm{II}_g=-\tfrac{1}{2}\dot{g}_0$ as a $(0,2)$-tensor field on $\dM$.
The associated Weingarten map of $\dM$ is denoted by $W_g$.
Its mean curvature is given by
\begin{equation}\nonumber
H_g=\tfrac{1}{n-1}\tr(W_g)=-\tfrac{1}{2(n-1)}\tr_{g_0}(\dot{g}_0).
\end{equation}

More generally, for each $(t,p)\in V$, we find a neighbourhood $p\in V_p\subset\dM$ such that $\{t\}\times V_p\subset V$.
The second fundamental form for this slice is $\mathrm{II}_t=-\frac{1}{2}\dot{g}_t$ (with respect to the normal $\frac{\partial}{\partial t}$).
They build a map $\mathrm{II}_\bullet\colon V\to T^\ast\dM\otimes T^\ast\dM$.
So do the corresponding Weingarten maps $W_\bullet\colon V\to T^\ast\dM\otimes T(\dM)$, uniquely defined by
\begin{equation}\nonumber
\langle W_t(p)(v),w\rangle_{g_t}=\mathrm{II}_t(p)(v,w)=-\frac{1}{2}\dot{g}_t(p)(v,w).
\end{equation}
The scalar curvature of $(M,g)$ within $U^g$ can be computed by means of \cite[Prop.~4.1]{BGM}:
\begin{equation}\label{magicformula}
\scal_g=\scal_{g_t}+3\,\tr(W_t^2)-\tr(W_t)^2-\tr_{g_t}(\ddot{g}_t).
\end{equation}
\begin{definition}
Let $C\in C^{\infty}(\dM;\mathbb{R})$.
A Riemannian metric $g\in\mathscr{R}(M)$ is called \emph{C-normal} if the pullback metrics $g_\bullet$ with respect to some small geodesic collar neighbourhood $\varrho_g\colon V\to U^g$ are given by
\begin{align}\nonumber
g_t(p)&=g_0(p)+t\cdot\dot{g}_0(p)-C(p)t^2\cdot g_0(p)\\\nonumber
&=g_0(p)-2t\cdot\mathrm{II}_g(p)-C(p)t^2\cdot g_0(p),\;(t,p)\in V.
\end{align}
\end{definition}
Every Riemannian metric on $M$ can be deformed into a $C$-normal metric while preserving lower scalar curvature bounds, with the original $1$-jet along the boundary:
\begin{proposition}\label{proposition23}
Let $K$ be a compact Hausdorff space and let $g\colon K\to\mathscr{R}_{>\sigma}(M)$ be a continuous family of Riemannian metrics of scalar curvature greater than $\sigma$.

Then there exists a smooth function $C_0\in C^{\infty}(\dM;\mathbb{R})$ such that for each function $C\in C^{\infty}(\partial M;\mathbb{R})$ with $C\geq C_0$ and each neighbourhood $\mathscr{U}$ of $\dM$, there exists a continuous map
\begin{equation}\nonumber
f\colon K\times[0,1]\to\mathscr{R}_{>\sigma}(M)
\end{equation}
such that the following holds for all $\xi\in K$ and $s\in[0,1]$:
\begin{enumerate}
\item[(a)]{$f(\xi,0)=g(\xi)$;}
\item[(b)]{$f(\xi,1)$ is $C$-normal;}
\item[(c)]{if $g(\xi)$ is $\tilde{C}$-normal, then $f(\xi,s)$ is $((1-s)\tilde{C}+sC)$-normal;}
\item[(d)]{$f(\xi,s)_0=g(\xi)_0$ and $\mathrm{II}_{f(\xi,s)}=\mathrm{II}_{g(\xi)}$;}
\item[(e)]{$\ddot{f}(\xi,s)_0=(1-s)\ddot{g}(\xi)_0-2sCg(\xi)_0$;}
\item[(f)]{$f(\xi,s)^{(\ell)}_0=(1-s)g(\xi)_0^{(\ell)}$ for all $\ell\geq 3$;}
\item[(g)]{$f(\xi,s)=g(\xi)$ on $M\setminus\mathscr{U}$;}
\item[(h)]{$f(\xi,s)$ is quasi-isometric to $g(\xi)$ via the identity;}
\item[(i)]{if $g(\xi)$ is complete, then so is $f(\xi,s)$.}
\end{enumerate}
\end{proposition}
\begin{proof}
Based on Corollary~\ref{uniformization}, it can be assumed that all metrics $g(\xi),\xi\in K$ have the same normal exponential map.

Let $\mathscr{U}$ be a neighbourhood of $\dM$ and let $\varrho\colon V\to U$ be a common geodesic collar neighbourhood for all metrics in $g$ satisfying $U\subset\mathscr{U}$.
One can identify the metrics $g(\xi),\xi\in K$ on $U$ with their associated generalised cylinder metrics on $V$.

Let $(U_i)_{i\in I}$ be a cover of $\dM$ consisting of open and relatively compact subsets, and let $(\psi_i)_{i\in I}$ be a partition of unity subordinate to $(U_i)_{i\in I}$.
For each $i\in I$, we define
\begin{equation}\nonumber
C_i:=\tfrac{1}{2(n-1)}\max_{\xi\in K}\Vert\tr_{g(\xi)_0}(\ddot{g}(\xi)_0)\Vert_{C^0(U_i)}.
\end{equation}
It holds $C_i<\infty$ due to compactness of $\cl(U_i)$.
Then put
\begin{equation}\nonumber
C_0\colon\dM\to\mathbb{R}\,,\,C_0:=\sum_{i\in I}\psi_i\cdot C_i.
\end{equation}
This is a smooth function with
\begin{equation}\nonumber
2(n-1)\cdot C_0(p)\geq|\tr_{g(\xi)_0(p)}(\ddot{g}(\xi)_0(p))|
\end{equation}
for all $\xi\in K$ and $p\in\partial M$.

Now let $C\colon\partial M\to\mathbb{R}$ be a smooth function with $C\geq C_0$.
We consider the Taylor expansion of $g_\bullet$ around $t=0$:
\begin{equation}\nonumber
g(\xi)_t=g(\xi)_0+\dot{g}(\xi)_0\cdot t+\frac{1}{2}\ddot{g}(\xi)_0\cdot t^2+R(\xi)_t.
\end{equation}
Here $R(\xi)_\bullet$ is a smooth map $R(\xi)_\bullet\colon V\to T^\ast\dM\otimes T^\ast\dM$ made out of symmetric $(0,2)$-tensors on $\dM$ that depends continuously on $\xi$.
It holds
\begin{equation}\nonumber
R(\xi)_0=\dot{R}(\xi)_0=\ddot{R}(\xi)_0=0
\end{equation}
for all $\xi\in K$.
Put
\begin{equation}\nonumber
F(\xi,s):=g(\xi)-s\left(\left(\tfrac{1}{2}\ddot{g}(\xi)_0+C\cdot g(\xi)_0\right)\cdot t^2+R(\xi)_t\right)
\end{equation}
for $\xi\in K$ and $s\in[0,1]$.
This defines a continuous family $F\colon K\to C^{\infty}(V;T^\ast V\otimes T^\ast V)$ of symmetric $(0,2)$-tensor fields.
It can also be seen as a family of $(0,2)$-tensor fields on $U$.

Shrinking $U$ if necessary, we can assume that $F(\xi,s)$ is a Riemannian metric on $U$ for each $\xi\in K$ and $s\in[0,1]$.
This is by continuity and compactness of $K$.
Furthermore, formula~\eqref{magicformula} gives
\begin{alignat}{2}\nonumber
\scal_{F(\xi,s)}|_{\dM}&=\;&&\scal_{F(\xi,s)_0}+3\,\tr(W_{F(\xi,s)}^2)-\tr(W_{F(\xi,s)})^2-\tr_{F(\xi,s)_0}(\ddot{F}(\xi,s)_0)\\\nonumber
&=\;&&\scal_{g(\xi)_0}+3\,\tr(W_{g(\xi)}^2)-\tr(W_{g(\xi)})^2-\tr_{g(\xi)_0}(\ddot{g}(\xi)_0)\\\nonumber
&&&-\tr_{g(\xi)_0}(-s\ddot{g}(\xi)_0-2sC\cdot g(\xi)_0)\\\nonumber
&=\;&&\scal_{g(\xi)}|_{\dM}+s\cdot\left(\tr_{g(\xi)_0}(\ddot{g}(\xi)_0)+2C\cdot(n-1)\right)\geq\scal_{g(\xi)}|_{\dM}>\sigma|_{\dM}.
\end{alignat}
Continuity allows to assume that $F(\xi,s)\in\mathscr{R}_{>\sigma}(U)$ for all $\xi\in K$ and $s\in[0,1]$.
Along the boundary, we find
\begin{enumerate}
\item[\myicon]{$F(\xi,s)|_{\dM}=g(\xi)|_{\dM}$;}
\item[\myicon]{$\dot{F}(\xi,s)_0=\dot{g}(\xi)_0$;}
\item[\myicon]{$\ddot{F}(\xi,s)_0=(1-s)\ddot{g}(\xi)_0-2sCg(\xi)_0$;}
\item[\myicon]{$F(\xi,s)_0^{(\ell)}=(1-s)g(\xi)_0^{(\ell)}$ for all $\ell\geq 3$.}
\end{enumerate}
Finally, since $F(\xi,s)|_{\dM}=g(\xi)|_{\dM}$, it can be arranged that
\begin{equation}\label{bilipschitzproposition23}
\frac{1}{2}|v|_{g(\xi)}\leq|v|_{F(\xi,s)}\leq 2|v|_{g(\xi)}
\end{equation}
for all $v\in TU,\xi\in K$ and $s\in[0,1]$. This time we use continuity for the maps
\begin{equation}\nonumber
K\times[0,1]\times\tilde{U}\times S^{n-1}\to\mathbb{R}\,,\,(\xi,s,p,v)\mapsto\frac{|v|_{F(\xi,s)(p)}}{|v|_{g(\xi)(p)}}
\end{equation}
where $\tilde{U}\subset U$ is an arbitraty coordinate neighbourhood.
Observe that one can restrict to the Euclidean sphere $S^{n-1}$ which is compact and independent from $\xi$.

The condition $\scal>\sigma$ defines a second order open partial differential relation on the space of pointwise metrics over $M$.
The metrics in $g$ solve it globally, the metrics in $F$ solve it locally over $U$.

By the family version of the local flexibility lemma \cite[Addendum 3.4]{BH2022} of Bär-Hanke, we obtain an open neighbourhood $\dM\subset U_0\subset U$ and a family 
\begin{equation}\nonumber
f\colon K\times[0,1]\to C^{\infty}(M;T^\ast M\otimes T^\ast M)
\end{equation}
such that the following holds for all $\xi\in K$ and $s\in[0,1]$:
\begin{enumerate}
\item[\myicon]{$f(\xi,s)\in\mathscr{R}_{>\sigma}(M)$;}
\item[\myicon]{$f(\xi,0)=g(\xi)$;}
\item[\myicon]{$f(\xi,s)|_{U_0}=F(\xi,s)|_{U_0};$}
\item[\myicon]{$f(\xi,s)|_{M\setminus U}=g(\xi)|_{M\setminus U}$.}
\end{enumerate}
In particular, $f(\xi,s)|_{M\setminus\mathscr{U}}=g(\xi)|_{M\setminus\mathscr{U}}$.
We conclude properties $(a)-(g)$ for this $f$.

The remaining properties $(h)-(i)$ can be deduced from the actual construction in the flexibility lemma:
Accordingly there exists a smooth function $\tau\colon M\to[0,1]$ such that
\begin{equation}\nonumber
f(\xi,s)(p)=\left\{\begin{array}{ll} F(\xi,s\tau(p))(p), & p\in U, \\
         g(\xi)(p), & \text{else}.\end{array}\right.
\end{equation}
Together with~\eqref{bilipschitzproposition23}, this yields
\begin{equation}\nonumber
\frac{1}{2}|v|_{g(\xi)}\leq|v|_{f(\xi,s)}\leq 2|v|_{g(\xi)}
\end{equation}
for all $v\in TM,\xi\in K$ and $s\in[0,1]$.
Properties $(h)-(i)$ follow immediately.
\end{proof}
\begin{remark}\label{samegeodesics}
The identification $g=\mathrm{d}t^2+g_t$ refers to a geodesic collar neighbourhood $\varrho_g\colon V\to U^g$ of $g$.
The deformations $F(\xi,s)$ are defined in terms of the collars $\varrho_{g(\xi)}$.
As soon as $F(\xi,s)$ is a Riemannian metric on $U=U^{g(\xi)}$ with its own geodesic collars, one may ask for the meaning of $F(\xi,s)_t$.
In fact, there is no subtlety because the normal geodesics of $F(\xi,s)$ and $g(\xi)$ coincide near the boundary.
\end{remark}
\begin{remark}
Given $L>1$, we can arrange the deformation so that
\begin{equation}\nonumber
\frac{1}{L}|v|_{g(\xi)}\leq|v|_{f(\xi,s)}\leq L|v|_{g(\xi)}.
\end{equation}
\end{remark}
\begin{remark}
Strictly speaking, Bär-Hanke's local flexibility lemma applies to smooth manifolds without boundary.
To obtain an admissible setting, we attach a small cylinder to $\dM$ and extend the metrics $g(\xi)$ to the new manifold -- called $M^+$ -- in continuous dependence on $\xi$.

Specifically, there is a continuous map $g^+\colon K\to\mathscr{R}(M^+)$ such that the canonical inclusion $(M,g(\xi))\to(M^+,g^+(\xi))$ is an isometric embedding for every $\xi\in K$.
We can guarantee continuity for $g^+$ by means of the Seeley extension theorem \cite{Seeley}.
See \cite[Prop.~2.8]{Frerichs} for details.
\end{remark}
The last deformation step is for adjusting the $1$-jet along the boundary, again respecting lower scalar curvature bounds.
Several lemmas are required.
\begin{lemma}
Let $V$ be a finite dimensional real vector space and let $g_1$ and $g_0$ be two Euclidean scalar products on $V$ such that $\Vert g_1-g_0\Vert_{g_0}\leq\frac{1}{2}$.
Then
\begin{equation}\nonumber
|\tr_{g_1}(h)-\tr_{g_0}(h)|\leq 2\cdot\Vert g_1-g_0\Vert_{g_0}\cdot\Vert h\Vert_{g_0}
\end{equation}
holds for all symmetric bilinear forms $h$ on $V$.
Here $\Vert\cdot\Vert_{g_0}$ denotes the Frobenius norm on the space of symmetric bilinear forms induced by $g_0$.
\end{lemma}
\begin{proof}
See \cite[Lem.~3.4]{BH2023}.
\end{proof}
\begin{lemma}\label{bumpfunction}
There exists a constant $c_0>0$ such that for each $0<\delta\leq\frac{1}{2}$ there exists a smooth function $\chi_\delta\colon[0,\infty)\to\mathbb{R}$ with
\begin{enumerate}
\item[\myicon]{$\chi_\delta(t)=t$ for $t$ near $0$, $\chi_\delta(t)=0$ for $t\geq\sqrt{\delta}$ and $0\leq\chi_\delta(t)\leq\frac{\delta}{2}$ for all $t$,}
\item[\myicon]{$|\dot{\chi}_\delta(t)|\leq c_0$ for all $t$,}
\item[\myicon]{$-\frac{2}{\delta}\leq\ddot{\chi}_\delta(t)\leq 0$ for all $t\in[0,\delta]$ and $|\ddot{\chi}_\delta(t)|\leq c_0$ for all $t\in[\delta,\sqrt{\delta}]$.}
\end{enumerate}
\end{lemma}
\begin{proof}
See \cite[Lem.~3.5]{BH2023}.
\end{proof}
\begin{remark}\label{bumpfunctionestimate}
Given $0<\delta\leq\frac{1}{2}$ and a function $\chi_\delta$ as in Lemma~\ref{bumpfunction}, it holds $\chi_\delta(t)\leq t$ for all $t\in[0,\infty)$:
Since $\ddot{\chi}_\delta\leq 0$ on $[0,\delta]$, we find $\dot{\chi}_\delta(t)\leq\dot{\chi}_\delta(0)=1$ for all $t\in[0,\delta]$. Thus, $\chi_\delta(t)\leq t$ by the mean value theorem.
For $t\in(\delta,\infty)$, we have $\chi_\delta(t)\leq\frac{\delta}{2}\leq t$.
\end{remark}
\begin{lemma}\label{cover2}
There exist countable covers $(U_i^1)_{i\in\mathbb{N}}\subset(U_i^2)_{i\in\mathbb{N}}\subset(U_i^3)_{i\in\mathbb{N}}$ of $\dM$ such that the following holds for all $i\in\mathbb{N}$:
\begin{enumerate}
\item[\myicon]{Each subset $U_i^1,U_i^2,U_i^3\subset\dM$ is open and relatively compact;}
\item[\myicon]{$I_i:=\{j\in\mathbb{N}:U_i^3\cap U_j^3\neq\varnothing\}$ is finite;}
\item[\myicon]{$\cl(U_i^1)\subset U_i^2$ and $\cl(U_i^2)\subset U_i^3$.}
\end{enumerate}
\end{lemma}
\begin{proof}
The proof is similar to that of Lemma~\ref{cover1}.
Let $\Phi\colon M\to\mathbb{R}^N$ be a proper smooth embedding into some Euclidean space.
Then the sets $U_i^k:=\Phi^{-1}(B(i,k\sqrt{N}))$ for $k=1,2,3$ and $i\in\mathbb{Z}^{N}\cong\mathbb{N}$ do the job.
\end{proof}
We fix three covers $(U_i^1)_{i\in\mathbb{N}}\subset (U_i^2)_{i\in\mathbb{N}}\subset (U_i^3)_{i\in\mathbb{N}}$ of $\dM$ as in Lemma~\ref{cover2} together with a partition of unity $\psi=(\psi_i)_{i\in\mathbb{N}}$ subordinate to $(U_i^1)_{i\in\mathbb{N}}$.

The next lemma concerns the construction of closed collar neighbourhoods in manifolds with boundary.
This is not a trivial matter, as demonstrated by the following example:
Let $g$ be a Riemannian metric on $M$.
Suppose there exists a geodesic collar $\varrho_g\colon[0,\varepsilon)\times\dM\to M$ of uniform width $\varepsilon>0$.

If $\dM$ is compact, then $\varrho_g([0,\frac{\varepsilon}{2}]\times\partial M)$ is compact, hence closed in $M$.
However, if $\dM$ is non-compact, then $\varrho_g([0,\frac{\varepsilon}{2}]\times\partial M)$ is not closed in general, as can be seen with $M=\mathbb{R}^2_{x_1\geq 0}\setminus(\{0\}\times\mathbb{R}_{\leq 0})$ and $g=g_{\mathrm{eucl}}$.
If the deformation was done within such a collar, the resulting metric would not be smooth and not even continuous along the $x_1$-axis.
Nevertheless, there does exist a closed collar in our example, namely the set $A=\{(x,y)\in\mathbb{R}^2:0<x\leq y\}$.
Another closed collar is given by
\begin{equation}\nonumber
B=\Bigl(\bigcup_{i\geq 2}\Bigl[0,\frac{1}{i}\Bigr]\times\Bigl[\frac{1}{i},\frac{1}{i-1}\Bigr]\Bigr)\cup\Bigl(\Bigl[0,\frac{1}{2}\Bigr]\times[1,\infty)\Bigr).
\end{equation}
We generalise the collar neighbourhood $B$ to arbitrary manifolds. 
\begin{figure}[h]
\begin{center}
\begin{tikzpicture}
\node{\includegraphics[width=0.6\textwidth]{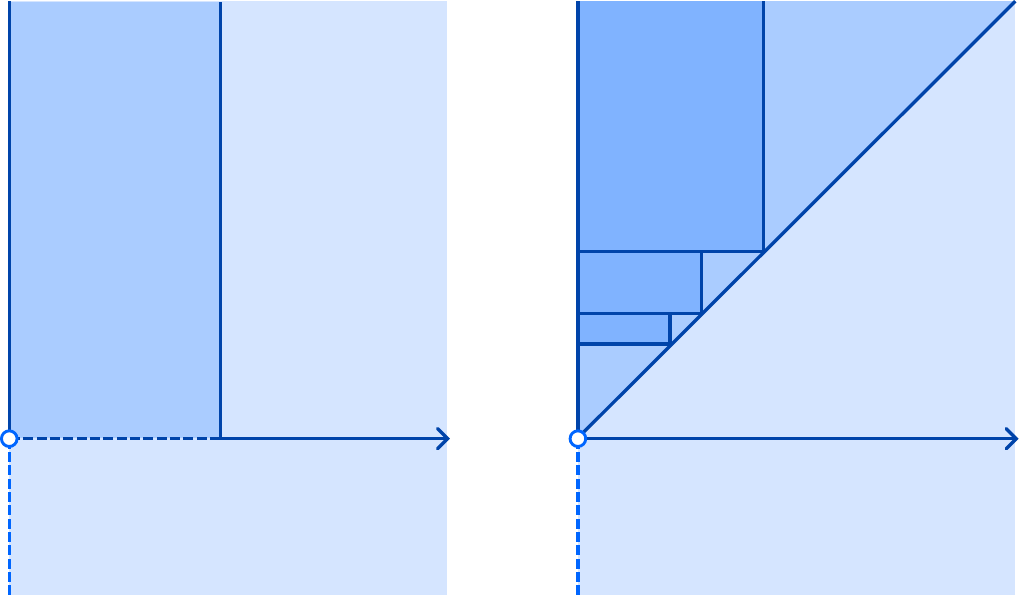}};
\node at (4.6,-1.6) {$x_1$};
\node at (1.5,1.6) {$B$};
\node at (3.2,2.2) {$A$};
\node at (-0.73,-1.6) {$x_1$};
\node at (1,2.2) {$\dM$};
\node at (-4.33,2.2) {$\dM$};
\node at (-3.7,-1.6) {$\varepsilon$};
\end{tikzpicture}
\end{center}
\caption{Collar neighbourhoods in $M=\mathbb{R}_{x_1\geq 0}^2\setminus(\{0\}\times\mathbb{R}_{\leq 0})$ -- $A$ and $B$ are closed, whereas the collar in the left picture is not}
\end{figure}
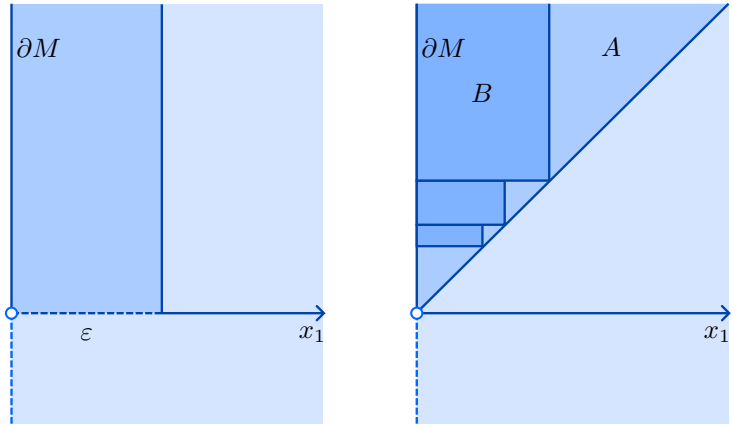

\begin{lemma}\label{lemma3.6}
Let $g$ be a Riemannian metric on $M$ and let $\varrho_g\colon V_{\eta}\to U^g_\eta$ be a geodesic collar neighbourhood.
For each $i\in\mathbb{N}$, let $0<\varepsilon_i<\min\{i^{-1},\inf_{p\in U_i^3}\eta(p)\}$ and $U_{\smallblacksquare\varepsilon_i}^{1}:=\varrho_g([0,\varepsilon_i]\times\cl(U_i^1))$.
Then $\cup_{i\in\mathbb{N}}\,U_{\smallblacksquare\varepsilon_i}^{1}\subset M$ is a closed neighbourhoood of $\partial M$.
\end{lemma}
\begin{proof}
Let $(q_n)_{n\in\mathbb{N}}$ be a sequence in $\cup_{i\in\mathbb{N}}\,U_{\smallblacksquare\varepsilon_i}^{1}$ with $q_n\to q\in M$ for $n\to\infty$.
We show $q\in\cup_{i\in\mathbb{N}}\,U_{\smallblacksquare\varepsilon_i}^{1}$. 
To this end, let $(\varrho_g^{-1}(q_n))_n=(t_n,p_n)$ denote the associated sequence in the preimage.

\textit{First case:} There is $N\in\mathbb{N}$ such that $p_n\in\cup_{i\leq N}\cl(U_i^1)$ for all $n\in\mathbb{N}$.
Then it holds $q_n\in\cup_{j\in I_i,i\leq N}\,U_{\smallblacksquare\varepsilon_j}^{1}$ for all $n\in\mathbb{N}$.
The last set is closed as a finite union of compact sets. Therefore, $q\in\cup_{i\in\mathbb{N}}\,U_{\smallblacksquare\varepsilon_i}^{1}$.

\textit{Second case:} There is no such $N\in\mathbb{N}$. By the estimate $\varepsilon_i<i^{-1}$, we find a subsequence $t_{\varphi(n)}$ of $t_n$ such that $t_{\varphi(n)}\to 0$ for $n\to\infty$, that is $d_g(q_{\varphi(n)},\dM)\to 0$ for $n\to\infty$.
The triangle inequality gives
\begin{equation}\nonumber
d_g(q,\dM)\leq d_g(q,q_{\varphi(n)})+d_g(q_{\varphi(n)},\dM)\quad\text{for all $n\in\mathbb{N}$}
\end{equation}
so that $d_g(q,\dM)=0$. Since $\dM\subset M$ is closed, this yields $q\in\dM\subset\cup_{i\in\mathbb{N}}\,U_{\smallblacksquare\varepsilon_i}^{1}$.
\end{proof}
\begin{proposition}\label{proposition26}
Let $K$ be a compact Hausdorff space. 
Let $g_0\colon K\to C^{\infty}(\dM;T^\ast\dM\otimes T^\ast\dM)$ be a continuous family of Riemannian metrics on $\dM$ and let $h,k\colon K\to C^{\infty}(\dM;T^\ast\dM\otimes T^\ast\dM)$ be continuous families of symmetric $(0,2)$-tensor fields satisfying $\tr_{g_0}(h)\geq\tr_{g_0}(k)$.

Then there exists a smooth function $C_0=C_0(g,h,k)\in C^{\infty}(\dM;\mathbb{R}),C_0>0$ such that
\begin{enumerate}
\item[\myicon]{for every continuous family
\begin{equation}\nonumber
g\colon K\to\mathscr{R}_{>\sigma}(M)
\end{equation}
of $C$-normal metrics of scalar curvature greater than $\sigma$ with $C\in C^{\infty}(\dM;\mathbb{R}),C\geq C_0, g(\xi)_0=g_0(\xi)$ and $\mathrm{II}_{g(\xi)}=h(\xi)$ for all $\xi\in K$ and}
\item[\myicon]{for each neighbourhood $\mathscr{U}$ of $\dM$}
\end{enumerate}
there exists a continuous map
\begin{equation}\nonumber
f\colon K\times[0,1]\to\mathscr{R}_{>\sigma}(M)
\end{equation}
so that the following holds for all $\xi\in K$ and $s\in[0,1]$:
\begin{enumerate}
\item[(a)]{$f(\xi,0)=g(\xi)$;}
\item[(b)]{$f(\xi,s)$ is $C$-normal;}
\item[(c)]{$f(\xi,s)_0=g(\xi)_0$;}
\item[(d)]{$\mathrm{II}_{f(\xi,s)}=(1-s)\mathrm{II}_{g(\xi)}+sk(\xi)$;}
\item[(e)]{$f(\xi,s)=g(\xi)$ on $M\setminus\mathscr{U}$;}
\item[(f)]{$f(\xi,s)$ is quasi-isometric to $g(\xi)$ via the identity;}
\item[(g)]{if $g(\xi)$ is complete, then so is $f(\xi,s)$.}
\end{enumerate}
\end{proposition}
\begin{proof}
Let $\mathscr{U}$ be a neighbourhood of $\dM$.
Let $C\in C^{\infty}(\dM;\mathbb{R})$ be a positive smooth function and let $g\colon K\to\mathscr{R}_{>\sigma}(M)$ be a continuous family of $C$-normal metrics of scalar curvature greater than $\sigma$ with $g(\xi)_0=g_0(\xi)$ and $\mathrm{II}_{g(\xi)}=h(\xi)$ for all $\xi\in K$.

Given that this proposition follows Corollary~\ref{uniformization} and Proposition~\ref{proposition23} within the deformation principle, we can assume that all metrics in $g$ have the same normal exponential map and that they satisfy the condition of $C$-normality on a common neighbourhood $U_0$ of $\dM$.

We choose a continuous positive function $\eta\colon\dM\to\mathbb{R}$ such that $\varrho\colon V_\eta\to U_\eta$ is a common geodesic collar neighbourhood for $g$ with $U_\eta\subset U_0\cap\mathscr{U}$.
For each $i\in\mathbb{N}$, let $0<\eta_i<\inf_{p\in U_i^3}\eta(p)$ be some fixed distance from the boundary.

Later, we will consider the $C^2$-norm $\Vert\cdot\Vert_{C^2}$ for tensor fields on $\dM$, which is supposed to be the maximum of $C^2$-norms induced by the metrics $g_0(\xi),\xi\in K$.
Set
\begin{equation}\nonumber
C_i:=\Vert C\Vert_{C^2(U_i^3)}.
\end{equation}
The deformation will be performed over each piece $U_i^2\subset\dM$, gluing everything together by the partition of unity $\psi$.
To this end, let $\delta=(\delta_i)_{i\in\mathbb{N}}$ be a family of real numbers satisfying
\begin{equation}\nonumber
0<\delta_i<\min\left\{\frac{1}{2},i^{-2},\min_{j\in I_i}\eta_j^2,\min_{j\in I_i}C_j^{-2}\right\}
\end{equation}
for all $i\in\mathbb{N}$.
For each $i\in\mathbb{N}$, let $\chi_{\delta_i}$ be a function as in Lemma~\ref{bumpfunction}.
We define
\begin{align}\nonumber
f^{\delta_i}\colon&K\times[0,1]\to C^{\infty}(U_{\smallsquare\eta_i}^2;T^\ast M\otimes T^\ast M),\\\nonumber
&f^{\delta_i}(\xi,s)=\mathrm{d}t^2+(1-Ct^2)\cdot g_0(\xi)-2t\cdot h(\xi)+2s\chi_{\delta_i}(t)\cdot(h(\xi)-k(\xi))
\end{align}
where $U_{\smallsquare\eta_i}^2:=\varrho([0,\eta_i)\times U_i^2)$.
Given $(t,p)\in[0,\eta_i)\times U_i^2$, it holds
\begin{align}\nonumber
\Vert f^{\delta_i}(\xi,s)(t,p)-g(\xi)(t,p)\Vert_{g(\xi)(t,p)}&=2s\chi_{\delta_i}(t)\cdot\Vert h(\xi)(p)-k(\xi)(p)\Vert_{g(\xi)_t(p)}\\\nonumber
&\leq\delta_i\cdot\sup_{\substack{\xi\in K,t\in[0,\eta_i]\\p\in U_i^2}}\Vert h(\xi)(p)-k(\xi)(p)\Vert_{g(\xi)_t(p)}.
\end{align}
Since positive definiteness is an open condition, $f^{\delta_i}$ becomes a family of Riemannian metrics on $U_{\smallsquare\eta_i}^2$ for sufficiently small $\delta_i$.
Then put
\begin{align}\nonumber
f^\delta\colon&K\times[0,1]\to C^{\infty}(M;T^\ast M\otimes T^\ast M),\\\nonumber
&f^\delta(\xi,s)=\left\{\begin{array}{l}\mathrm{d}t^2+(1-Ct^2)\cdot g_0(\xi)-2t\cdot h(\xi)+\sum\limits_{i\in\mathbb{N}}\psi_i\cdot 2s\chi_{\delta_i}(t)\cdot(h(\xi)-k(\xi))\\
\phantom{g(\xi)}\qquad\text{on $\bigcup\limits_{i\in\mathbb{N}}U_{\smallsquare\eta_i}^2$,}\\
         g(\xi)\qquad\text{on $M-\bigcup\limits_{i\in\mathbb{N}}U_{\smallsquare\eta_i}^2$}\end{array}\right.
\end{align}
for $\delta=(\delta_i)_{i\in\mathbb{N}}$ as small as above.
This expression is positive definite at each point $(t,p)\in\bigcup_{i\in\mathbb{N}}U_{\smallsquare\eta_i}^2$, as can be seen from the formula
\begin{equation}\label{001}
f^\delta(\xi,s)=\sum\limits_{j\in\mathbb{N}}\psi_j\cdot(\mathrm{d}t^2+(1-Ct^2)\cdot g_0(\xi)-2t\cdot h(\xi)+2s\chi_{\delta_j}(t)\cdot(h(\xi)-k(\xi))).
\end{equation}
If $p\in U_j^2$ for some $j\in\mathbb{N}$ and $t<\eta_j$, then the respective summand is positive definite by the above.
If $t\geq\eta_j$, then $\chi_{\delta_j}(t)=0$ and the other three summands are equal to $g(\xi)(t,p)$.
There are no more summands to take into account because $\psi_j(p)=0$ for all $j\in\mathbb{N}$ with $p\notin U_j^2$.

We also verify the smoothness of $f^\delta(\xi,s)$ and the continuity of $f^\delta$ with respect to $(\xi,s)$.
This is obvious once we know
\begin{equation}\nonumber
f^\delta(\xi,s)=g(\xi)\quad\text{on $M-\bigcup_{i\in\mathbb{N}}\hspace{1pt}U_{\smallblacksquare\sqrt{\delta_i}}^1$},
\end{equation}
the latter beeing open in $M$ by Lemma~\ref{lemma3.6}.
To justify, let $(t,p)\in\bigcup_{i\in\mathbb{N}}U_{\smallsquare\eta_i}^2-\bigcup_{i\in\mathbb{N}}\hspace{1pt}U_{\smallblacksquare\sqrt{\delta_i}}^1$.
That is $(t,p)\in U_{\smallsquare\eta_i}^2$ for some $i\in\mathbb{N}$.
Every summand in~\eqref{001} is either $\psi_j(p)\cdot g(\xi)(t,p)$ (if $p\in U_j^1$) or zero (if $p\notin U_j^1$).
Therefore, $f^\delta(\xi,s)(t,p)=\sum_{j\in\mathbb{N}}\psi_j(p)\cdot g(\xi)(t,p)=g(\xi)(t,p)$.

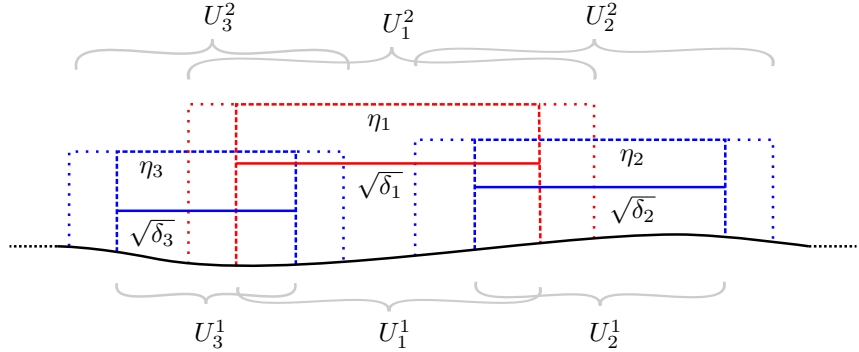
\begin{figure}[h]
\begin{tikzpicture}
\node{\includegraphics[width=0.7\textwidth]{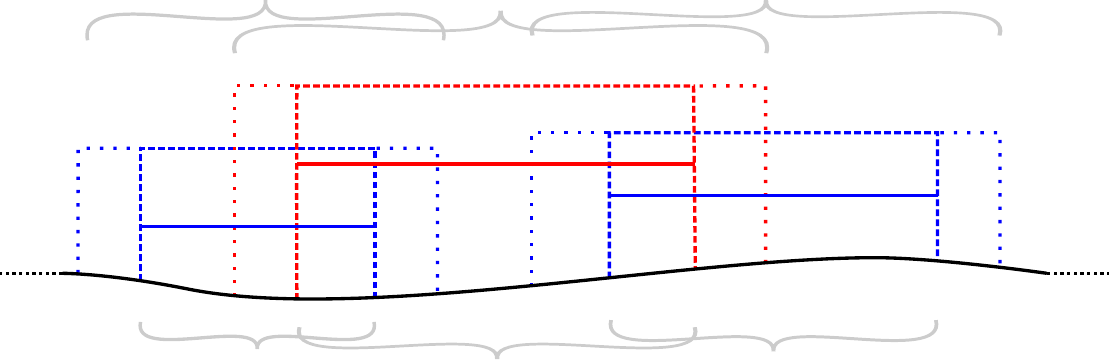}};
\node at (-0.5,-2.1) {$U_1^1$};
\node at (-0.7,-0.125) {$\sqrt{\delta_1}$};
\node at (-0.7,0.7) {$\eta_1$};
\node at (-0.45,2) {$U_1^2$};
\node at (2.275,-2.1) {$U_2^1$};
\node at (2.225,2.1) {$U_2^2$};
\node at (2.625,-0.45) {$\sqrt{\delta_2}$};
\node at (2.625,0.225) {$\eta_2$};
\node at (-2.925,-2.1) {$U_3^1$};
\node at (-2.8,2.1) {$U_3^2$};
\node at (-3.7,-0.76) {$\sqrt{\delta_3}$};
\node at (-3.7,0.075) {$\eta_3$};
\end{tikzpicture}
\caption{Construction for $f^\delta$ by boxes}
\end{figure}

The main task is to arrange the deformation so that $\scal_{f^\delta(\xi,s)}>\sigma$ for all $\xi\in K$ and $s\in[0,1]$.
Let $i\in\mathbb{N}$.
In the following we use the notation $\lesssim$ to mean that the left hand side is bounded by the right hand side multiplied with a positive constant that only depends on $g_0,h$ and $k$ on $U_i^3$.
In particular, it is independent from $\delta,t,p,\xi,s$ and $C$.

Furthermore, a statement is said to be satisfied for sufficiently small $(\delta_j)_{j\in I_i}$ if it is satisfied for all tuples $(\delta_j)_{j\in I_i}$ whose maximum is smaller than a certain constant.
This constant depends only on $g_0,h$ and $k$ on $U_i^3$ and on $C$ on $\bigcup_{j\in I_i}U_j^3$.

On $U_i^3$ we use the abbreviation
\begin{align}\nonumber
&\gamma:=f^\delta(\xi,s)\quad\text{and}\\\nonumber
&\gamma_t:=f^\delta(\xi,s)_t=(1-Ct^2)\cdot g_0(\xi)-2t\cdot h(\xi)+\sum_{j\in I_i}\psi_j\cdot 2s\chi_{\delta_j}(t)\cdot(h(\xi)-k(\xi))
\end{align}
for $t\in[0,\eta_i)$.
It holds
\begin{align}\nonumber
\scal_\gamma&=\scal_{\gamma_t}+3\tr(W_t^2)-\tr(W_t)^2-\tr_{\gamma_t}(\ddot{\gamma}_t)\\\label{3summands}
&\geq\scal_{\gamma_t}-\tr_{\gamma_t}(\mathrm{II}_t)^2-\tr_{\gamma_t}(\ddot{\gamma}_t)
\end{align}
since $W_t$ is a field of self-adjoint endomorphisms.
Moreover,
\begin{align}\nonumber
&\mathrm{II}_t=-\frac{1}{2}\dot{\gamma}_t=h(\xi)+Ct\cdot g_0(\xi)-\sum_{j\in I_i}\psi_j\cdot s\dot{\chi}_{\delta_j}(t)\cdot(h(\xi)-k(\xi)),\\\nonumber
&\ddot{\gamma}_t=-2C\cdot g_0(\xi)+\sum_{j\in I_i}\psi_j\cdot 2s\ddot{\chi}_{\delta_j}(t)\cdot(h(\xi)-k(\xi)).
\end{align}
Note that normal geodesics of $f^\delta(\xi,s)$ and $g(\xi)$ coincide (cf. Remark~\ref{samegeodesics}).
We investigate each summand in~\eqref{3summands} on the set $[0,\eta_i)\times U_i^2$.
One can even restrict to $[0,\max_{j\in I_i}\sqrt{\delta_j})\times U_i^2$, since $\gamma$ is equal to $g(\xi)$ on $[\max_{j\in I_i}\sqrt{\delta_j},\eta_i)\times U_i^2$.

\textit{Regarding $\scal_{\gamma_t}$}:
For every $t\in[0,\max_{j\in I_i}\sqrt{\delta_j})$, it holds
\begin{align}\nonumber
\Vert\gamma_t-\gamma_0\Vert_{C^2(U_i^3)}&=\Bigl\Vert-Ct^2\cdot\gamma_0-2t\cdot h(\xi)+\sum_{j\in I_i}\psi_j\cdot 2s\chi_{\delta_j}(t)\cdot(h(\xi)-k(\xi))\Bigr\Vert_{C^2(U_i^3)}\\\nonumber
&\leq t^2\cdot\Vert C\cdot\gamma_0\Vert_{C^2(U_i^3)}+2t\cdot\Vert h(\xi)\Vert_{C^2(U_i^3)}+2s\sum_{j\in I_i}\chi_{\delta_j}(t)\cdot\Vert\psi_j\cdot(h(\xi)-k(\xi))\Vert_{C^2(U_i^3)}\\\nonumber
&\leq 4t^2\cdot C_i\cdot\Vert\gamma\Vert_{C^2(U_i^3)}+2t\cdot\Vert h(\xi)\Vert_{C^2(U_i^3)}+8s\sum_{j\in I_i}\chi_{\delta_j}(t)\cdot\Vert\psi_j\Vert_{C^2(U_i^3)}\cdot\Vert h(\xi)-k(\xi)\Vert_{C^2(U_i^3)}.
\end{align}
We have
\begin{equation}\nonumber
t^2\cdot C_i\leq t\cdot\max_{j\in I_i}\sqrt{\delta_j}\cdot C_i\leq t.
\end{equation}
Together with the inequality $2s\chi_{\delta_j}(t)\leq 2t$ from Remark~\ref{bumpfunctionestimate}, we find
\begin{align}\nonumber
\Vert\gamma_t-\gamma_0\Vert_{C^2(U_i^3)}&\leq 4t\cdot\Vert\gamma_0\Vert_{C^2(U_i^3)}+2t\cdot\Vert h(\xi)\Vert_{C^2(U_i^3)}+8t\sum_{j\in I_i}\Vert\psi_j\Vert_{C^2(U_i^3)}\cdot\Vert h(\xi)-k(\xi)\Vert_{C^2(U_i^3)}.
\end{align}
Hence $\Vert\gamma_t-\gamma_0\Vert_{C^2(U_i^3)}\lesssim t+t+t=3t\leq 3\max_{j\in I_i}\sqrt{\delta_j}$.
This yields
\begin{equation}\label{1}
\scal_{\gamma_t}\gtrsim -1\quad\text{on $U_i^2$}
\end{equation}
for sufficiently small $(\delta_j)_{j\in I_i}$.

\textit{Regarding $\tr_{\gamma_t}(\mathrm{II}_t)$}:
Let $(\delta_j)_{j\in I_i}$ be sufficiently small so that $\Vert\gamma_t-\gamma_0\Vert_{\gamma_0}\leq\frac{1}{2}$ for all $(t,p)\in[0,\max_{j\in I_i}\sqrt{\delta_j})\times U_i^2$.
It holds
\begin{align}\nonumber
|\tr_{\gamma_t}(\mathrm{II}_t)|&\leq|\tr_{\gamma_0}(\mathrm{II}_t)|+|\tr_{\gamma_t}(\mathrm{II}_t)-\tr_{\gamma_0}(\mathrm{II}_t)|\\\nonumber
&\lesssim|\tr_{\gamma_0}(\mathrm{II}_t)|+\Vert\gamma_t-\gamma_0\Vert_{\gamma_0}\cdot\Vert\mathrm{II}_t\Vert_{\gamma_0}\\\nonumber
&\lesssim|\tr_{\gamma_0}(\mathrm{II}_t)|+\Vert\mathrm{II}_t\Vert_{\gamma_0}.
\end{align}
By the estimate $|\dot{\chi}_{\delta_j}(t)|\leq c_0$ (cf. Lemma~\ref{bumpfunction}), we obtain
\begin{align}\nonumber
\Vert\mathrm{II}_t\Vert_{\gamma_0}&=\Vert h(\xi)+Ct\cdot\gamma_0-\sum_{j\in I_i}\psi_j\cdot s\dot{\chi}_{\delta_j}(t)\cdot(h(\xi)-k(\xi))\Vert_{\gamma_0}\\\nonumber
&\leq\Vert h(\xi)\Vert_{\gamma_0}+\sqrt{n-1}\cdot Ct+\sum_{j\in I_i}\psi_j\cdot s|\dot{\chi}_{\delta_j}(t)|\cdot\Vert h(\xi)-k(\xi)\Vert_{\gamma_0}\\\nonumber
&\lesssim 1+C_i\cdot\max_{j\in I_i}\sqrt{\delta_j}+\sum_{j\in I_i}\psi_j\\\nonumber
&\leq 3
\end{align}
and $|\tr_{\gamma_0}(\mathrm{II}_t)|\leq\sqrt{n-1}\cdot\Vert\mathrm{II}_t\Vert_{\gamma_0}\lesssim 1$.
Hence
\begin{equation}\label{2}
|\tr_{\gamma_t}(\mathrm{II}_t)|\lesssim|\tr_{\gamma_0}(\mathrm{II}_t)|+\Vert\mathrm{II}_t\Vert_{\gamma_0}\lesssim 1.
\end{equation}
\textit{Regarding $-\tr_{\gamma_t}(\ddot{\gamma}_t)$}:
The trace is given by
\begin{equation}\label{trace}
-\tr_{\gamma_t}(\ddot{\gamma}_t)=-\sum_{j\in I_i}\psi_j\cdot 2s\ddot{\chi}_{\delta_j}(t)\cdot\tr_{\gamma_t}(h(\xi)-k(\xi))+2C\cdot\tr_{\gamma_t}(\gamma_0).
\end{equation}
We consider the first sum:
Let $j_0\in I_i$ be fixed.
Since $\tr_{\gamma_0}(h)\geq\tr_{\gamma_0}(k),\ddot{\chi}_{\delta_{j_0}}(t)\leq 0$ and $\Vert\gamma_t-\gamma_0\Vert_{\gamma_0}\lesssim t$ for all $(t,p)\in[0,\delta_{j_0}]\times U_i^2$, it holds
\begin{align}\nonumber
\ddot{\chi}_{\delta_{j_0}}(t)\cdot\tr_{\gamma_t}(h(\xi)-k(\xi))&\leq\ddot{\chi}_{\delta_{j_0}}(t)\cdot\tr_{\gamma_t}(h(\xi)-k(\xi))-\ddot{\chi}_{\delta_{j_0}}(t)\cdot\tr_{\gamma_0}(h(\xi)-k(\xi))\\\nonumber
&\leq|\ddot{\chi}_{\delta_{j_0}}(t)|\cdot|\tr_{\gamma_t}(h(\xi)-k(\xi))-\tr_{\gamma_0}(h(\xi)-k(\xi))|\\\nonumber
&\lesssim|\ddot{\chi}_{\delta_{j_0}}(t)|\cdot\Vert\gamma_t-\gamma_0\Vert_{\gamma_0}\cdot\Vert h(\xi)-k(\xi)\Vert_{\gamma_0}\\\nonumber
&\lesssim\frac{2}{\delta_{j_0}}\cdot t\\\nonumber
&\leq 2
\end{align}
for such $t$ and $p$.
This calculation is where the estimate $\chi_{\delta_{j_0}}(t)\leq t$ comes into play, contrary to \cite[Prop.~3.6]{BH2023}.

For $t\in[\delta_{j_0},\sqrt{\delta_{j_0}})$, we find
\begin{align}\nonumber
|\ddot{\chi}_{\delta_{j_0}}(t)\cdot\tr_{\gamma_t}(h(\xi)-k(\xi))|&\lesssim|\tr_{\gamma_t}(h(\xi)-k(\xi))|\\\nonumber
&\lesssim|\tr_{\gamma_0}(h(\xi)-k(\xi))|+\Vert\gamma_t-\gamma_0\Vert_{\gamma_0}\cdot\Vert h(\xi)-k(\xi)\Vert_{\gamma_0}\\\nonumber
&\lesssim 1+1=2.
\end{align}
Finally, $\ddot{\chi}_{\delta_{j_0}}(t)=0$ for $t\geq\sqrt{\delta_{j_0}}$.
In total, we have
\begin{equation}\nonumber
\ddot{\chi}_{\delta_{j_0}}(t)\cdot\tr_{\gamma_t}(h(\xi)-k(\xi))\lesssim 1
\end{equation}
on $[0,\max_{j\in I_i}\sqrt{\delta_j})\times U_i^2$ so that
\begin{equation}\label{31}
\sum_{j\in I_i}\psi_j\cdot 2s\ddot{\chi}_{\delta_j}(t)\cdot\tr_{\gamma_t}(h(\xi)-k(\xi))\lesssim 1.
\end{equation}
The second summand in~\eqref{trace} satisfies
\begin{align}\nonumber
|\tr_{\gamma_t}(\gamma_0)-n+1|&=|\tr_{\gamma_t}(\gamma_0)-\tr_{\gamma_0}(\gamma_0)|\\\nonumber
&\lesssim\Vert\gamma_t-\gamma_0\Vert_{\gamma_0}\cdot\Vert\gamma_0\Vert_{\gamma_0}\\\nonumber
&\lesssim\max_{j\in I_i}\sqrt{\delta_j}.
\end{align}
This means
\begin{equation}\nonumber
|\tr_{\gamma_t}(\gamma_0)-n+1|\leq\tfrac{1}{2}
\end{equation}
for sufficiently small $(\delta_j)_{j\in I_i}$, respectively
\begin{equation}\label{32}
2C\cdot\tr_{\gamma_t}(\gamma_0)\geq 2(n-\tfrac{3}{2})C.
\end{equation}
Equations ~\eqref{trace},~\eqref{31} and ~\eqref{32} give
\begin{equation}\label{3}
-\tr_{\gamma_t}(\ddot{\gamma}_t)\gtrsim C
\end{equation}
if $C$ is large enough on $U_i^2$.
By~\eqref{3summands},~\eqref{1},~\eqref{2} and ~\eqref{3}, we obtain
\begin{equation}\label{3sum}
\scal_{\gamma}\gtrsim C,
\end{equation}
provided that $C$ is large enough on $U_i^2$.

Equations~\eqref{3} and~\eqref{3sum} make a condition on $C$ on $U_i^2$.
If we carry out the same calculation for every box $U^2_{\smallsquare\eta_j},j\in\mathbb{N}$, this will add more conditions to $C$ on $U_i^2$.
In fact, only the finite number of boxes $U^2_{\smallsquare\eta_j},j\in I_i$ can make a contribution for $U_i^2$.
This is why there exists a smooth function $C_{00}\in C^{\infty}(\partial M)$ which is sufficiently large on all pieces $U_j^2,j\in\mathbb{N}$ in the sense of the above calculation.

Given $C\geq C_{00}$, we impose finitely many conditions on $(\delta_j)_{j\in I_i}$, and in particular on $\delta_i$, in order to obtain equation~\eqref{3sum}.
Repeating the procedure for all boxes $U^2_{\smallsquare\eta_j},j\in\mathbb{N}$ will add only finitely many conditions  on $\delta_i$.
This allows the following conclusion:

\textit{Let $C\in C^{\infty}(\dM;\mathbb{R}),C\geq C_{00}$.
There exists a sequence of positive constants $(m_i)_{i\in\mathbb{N}}=\break m_i(g_0|_{U_i^3},h|_{U_i^3},k|_{U_i^3})$ such that for each collection $\delta=(\delta_i)_{i\in\mathbb{N}}$ of sufficiently small numbers, it holds
\begin{equation}\nonumber
\scal_\gamma\geq m_i\cdot C\quad\text{on\;\,$[0,\max_{j\in I_i}\sqrt{\delta_j})\times U_i^2$ for all $i\in\mathbb{N}$.}
\end{equation}
}
Finally,
\begin{equation}\nonumber
C_0\colon\dM\to\mathbb{R}\;,\;C_0=C_{00}+\sum_{i\in\mathbb{N}}2\psi_i\cdot\max_{j\in I_i}\bigl(m_j^{-1}\cdot\sup_{[0,\eta_j]\times U_j^2}|\sigma|\bigr)
\end{equation}
is a function as required in the proposition.

It remains to arrange properties $(f)$ and $(g)$:
Given $L>1$, we will be able to choose $\delta=(\delta_i)_{i\in\mathbb{N}}$ small enough so that
\begin{equation}\nonumber
\frac{1}{L}|v|_{g(\xi)}\leq|v|_{f(\xi,s)}\leq L|v|_{g(\xi)}
\end{equation}
for all $v\in TM,\xi\in K$ and $s\in[0,1]$.
For this, let $i\in\mathbb{N}$.
Without great loss of generality, one can assume that each closure $\cl(U_{\smallsquare\eta_j}^2),j\in I_i$ is covered by a single chart.
Then we add the following condition on $\delta$:
\begin{equation}\nonumber
0<\delta_i<\min_{j\in I_i}\inf\{|v|_{g(\xi)(q)}^4:\xi\in K,q\in U^2_{\smallsquare\eta_j},v\in S^{n-1}\}.
\end{equation}
For every $q\in U^2_{\smallsquare\eta_j},v\in S^{n-1},\xi\in K$ and $s\in[0,1]$, it holds
\begin{equation}\nonumber
\frac{|v|_{f(\xi,s)}^2}{|v|_{g(\xi)}^2}=1+\frac{1}{|v|_{g(\xi)}^2}\sum_{j\in I_i}\psi_j\cdot 2s\chi_{\delta_j}(t)\cdot(h(\xi)-k(\xi))(v^\mathsf{T},v^\mathsf{T})
\end{equation}
where $v^{\mathsf{T}}$ is the tangential part of $v$ with respect to $\mathrm{d}\varrho$.
We estimate
\begin{align}\nonumber
\frac{|v|_{f(\xi,s)}^2}{|v|_{g(\xi)}^2}-1&\geq-\frac{1}{|v|_{g(\xi)}^2}\sum_{j\in I_i}\delta_j\cdot|(h(\xi)-k(\xi))(v^{\mathsf{T}},v^{\mathsf{T}})|\\\nonumber
&\geq-\sum_{j\in I_i}\sqrt{\delta_j}\cdot|(h(\xi)-k(\xi))(v^{\mathsf{T}},v^{\mathsf{T}})|\gtrsim-\sum_{j\in I_i}\sqrt{\delta_j}.
\end{align}
This yields
\begin{equation}\nonumber
\frac{|v|_{f(\xi,s)}^2}{|v|_{g(\xi)}^2}-1\geq-\Bigl(1-\frac{1}{L^2}\Bigr)
\end{equation}
for sufficiently small $(\delta_j)_{j\in I_i}$.
Hence $|v|_{f(\xi,s)}/|v|_{g(\xi)}\geq 1/L$ for such $(\delta_j)_{j\in I_i}$.
Similarly,
\begin{align}\nonumber
\frac{|v|_{f(\xi,s)}^2}{|v|_{g(\xi)}^2}-1&\leq\frac{1}{|v|_{g(\xi)}^2}\sum_{j\in I_i}\delta_j\cdot|(h(\xi)-k(\xi))(v^{\mathsf{T}},v^{\mathsf{T}})|\\\nonumber
&\lesssim\sum_{j\in I_i}\sqrt{\delta_j}.
\end{align}
Choosing $(\delta_j)_{j\in I_i}$ small enough, we obtain
\begin{equation}\nonumber
\frac{|v|_{f(\xi,s)}^2}{|v|_{g(\xi)}^2}-1\leq L^2-1
\end{equation}
respectively $|v|_{f(\xi,s)}/|v|_{g(\xi)}\leq L$.
\end{proof}
The results from Proposition~\ref{proposition23} and~\ref{proposition26} are combined in the following theorem, where we use two auxiliary functions $S_1,S_2\colon[0,1]\to[0,1]$ defined by
\begin{equation}\nonumber
S_1(t)=\left\{\begin{array}{ll} 1, & 0\leq t\leq\tfrac{1}{2}\\
         2(1-t), & \tfrac{1}{2}\leq t\leq 1\end{array}\right.\quad\text{and}\quad
S_2(t)=\left\{\begin{array}{ll} 1-2t, & 0\leq t\leq\tfrac{1}{2}\\
         0, & \tfrac{1}{2}\leq t\leq 1\end{array}\right..
\end{equation}
\begin{theorem}\label{theorem27}
Let $K$ be a compact Hausdorff space and let
\begin{equation}\nonumber
g\colon K\to\mathscr{R}_{>\sigma}(M)
\end{equation}
be continuous.
Let $k\colon K\to C^{\infty}(\dM;T^\ast\dM\otimes T^\ast\dM)$ be a continuous family of symmetric $(0,2)$-tensor fields satisfying $\tfrac{1}{n-1}\tr_{g_0}(k(\xi))\leq H_{g(\xi)}$ for all $\xi\in K$.

Then there exists a smooth positive function $C_0\in C^{\infty}(\dM)$ such that for each $C\in C^{\infty}(\dM)$ with $C\geq C_0$ and for each neighbourhood $\mathscr{U}$ of $\dM$, there is a continuous map
\begin{equation}\nonumber
f\colon K\times[0,1]\to\mathscr{R}_{>\sigma}(M)
\end{equation}
so that the following holds for all $\xi\in K$ and $s\in[0,1]$:
\begin{enumerate}
\item[(a)]{$f(\xi,0)=g(\xi)$;}
\item[(b)]{$f(\xi,1)$ is $C$-normal;}
\item[(c)]{$f(\xi,s)_0=g(\xi)_0$;}
\item[(d)]{$\mathrm{II}_{f(\xi,s)}=S_1(s)\mathrm{II}_{g(\xi)}+(1-S_1(s))k(\xi)$, in particular $\mathrm{II}_{f(\xi,1)}=k(\xi)$;}
\item[(e)]{if $g(\xi)$ is $\tilde{C}$-normal, then $f(\xi,s)$ is $C_s$-normal for $C_s=S_2(s)\tilde{C}+(1-S_2(s))C$;}
\item[(f)]{$\ddot{f}(\xi,s)_0=S_2(s)\ddot{g}(\xi)_0-2(1-S_2(s))Cg(\xi)_0$;}
\item[(g)]{for $\ell\geq 3$ we have $f(\xi,s)_0^{(\ell)}=S_2(s)\cdot g(\xi)_0^{(\ell)}$;}
\item[(h)]{$f(\xi,s)=g(\xi)$ on $M\setminus\mathscr{U}$;}
\item[(i)]{$f(\xi,s)$ is quasi-isometric to $g(\xi)$ via the identity;}
\item[(j)]{if $g(\xi)$ is complete, then so is $f(\xi,s)$.}
\end{enumerate}
\end{theorem}
\begin{proof}
The proof is the same as in \cite[Thm.~3.7]{BH2023}.
We apply Proposition~\ref{proposition23} and Proposition~\ref{proposition26} for sufficiently large $C$, one after the other, and concatenate the arising homotopies.
\end{proof}
\begin{remark}
Using Theorem~\ref{theorem27}, the results \cite[Thm.~4.1~\&~4.11]{BH2023} of Bär-Hanke carry over to manifolds with non-compact boundary.
\end{remark}

\section{Applications}
The last section is about non-existence results for metrics with (uniformly) positive scalar curvature and mean convex boundary.
\subsection{Non-existence results for half-spaces}
We start with an elementary case to demonstrate the proof scheme.
\begin{corollary}\label{R2}
The right half-plane $\mathbb{R}^2_{x_1\geq 0}$ cannot carry any complete Riemannian metric of uniformly positive scalar curvature which has mean convex boundary.
\end{corollary}
\begin{proof}
Set $M:=\mathbb{R}^2_{x_1\geq 0}$.
Assume there is a Riemannian metric $g$ as in the theorem.
We apply Theorem~\ref{theorem27} to obtain a complete doubling metric $f\in\mathscr{R}(M)$ of uniformly positive scalar curvature.

Set $F:=f\cup f\in\mathscr{R}(\mathsf{D}M^f)$.
This metric is complete by Theorem~\ref{completemetrics} and it has uniformly positive scalar curvature.
Since $\mathsf{D}M^f$ is diffeomorphic to $\mathbb{R}^2$, we obtain a complete metric of uniformly positive scalar curvature on $\mathbb{R}^2$.
However, such a metric does not exist by the Bonnet-Myers theorem.
\end{proof}
In fact, there do exist metrics of positive scalar curvature and mean convex boundary on $\mathbb{R}_{x_1\geq 0}^2$.
An example of this is the following Riemannian hypersurface with boundary, applied in the case $n=2$:
\begin{equation}\nonumber
P_n:=\left\{x_{n+1}=\frac{1}{1-x_1^2-\ldots-x_n^2},x_1^2+\ldots+x_n^2<1,x_1\geq 0\right\}\subset(\mathbb{R}^{n+1},g_{\mathrm{eucl}}).
\end{equation}
Figure \ref{lastfigure} shows $P_2$ in $\mathbb{R}^3$.
For $n\geq 3$, the level sets of $P_n$ are half spheres of dimension two or greater, indicating that $P_n$ even has uniformly positive scalar curvature.
This is why Corollary~\ref{R2} cannot hold in higher dimensions.
However, if we fix a quasi-isometry type, we can find an analogous result for $n\geq 3$.
It requires a preparatory lemma.

\begin{figure}[h]
\begin{tikzpicture}
\node{\includegraphics[width=0.8\textwidth]{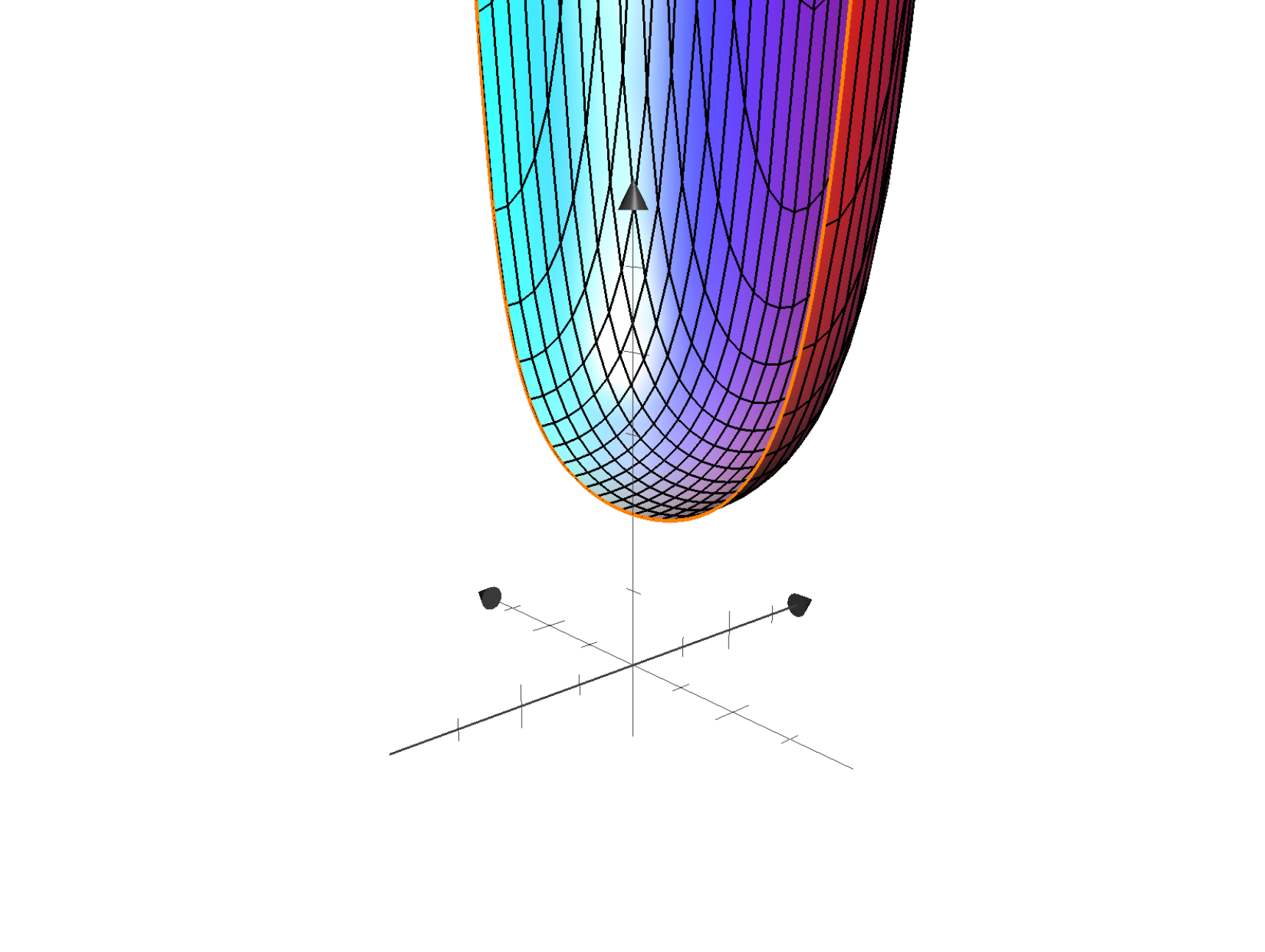}};
\node at (-0.105,-2.12) {$0$};
\node at (1,-2) {$0{,}5$};
\node at (0.45,-2.6) {$-0{,}5$};
\node at (-1,-2.8) {$-0{,}5$};
\node at (-0.105,-0.55) {$1$};
\node at (1.9,-1.6) {$x_1$};
\node at (-1.5,-1.6) {$x_2$};
\node at (0,3.2) {$x_3$};
\node at (-1.6,2) {$P_2$};
\end{tikzpicture}
\vspace{-1.8cm}
\caption{$P_2$ in $\mathbb{R}^3$}
\label{lastfigure}
\end{figure}
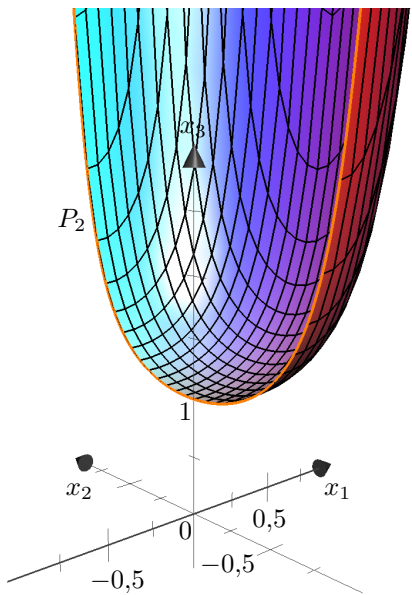

\begin{lemma}\label{quasiisometrydouble}
Let $M$ be a smooth manifold with non-empty boundary.
Let $g_1$ and $g_2$ be two complete doubling metrics on $M$.
Put $G_1:=g_1\cup g_1$ and $G_2:=g_2\cup g_2$.
If $\id\colon(M,d_{g_1})\to(M,d_{g_2})$ is a quasi-isometry, then so is $\id\colon(\mathsf{D}M,d_{G_1})\to(\mathsf{D}M,d_{G_2})$.
\end{lemma}
\begin{proof}
We assume that $M$ is connected and that $\id\colon(M,d_{g_1})\to(M,d_{g_2})$ is a quasi-isometry.
Then there exist constants $A\geq 1$ and $B\geq 0$ such that
\begin{equation}\nonumber
\frac{1}{A}d_{g_1}(p,q)-B\leq d_{g_2}(p,q)\leq A\,d_{g_1}(p,q)+B\quad\text{for all $p,q\in M$.}
\end{equation}
Let $p,q\in\mathsf{D}M$.
If $p$ and $q$ are contained in the same copy of $M$, then Proposition~\ref{samedistance} yields the desired quasi-isometry condition.

We focus on the case where $p$ and $q$ are contained in different copies.
Since $G_2$ is complete (see Theorem~\ref{completemetrics}), there exists a length-minimizing $G_2$-geodesic $\gamma\colon[0,1]\to\mathsf{D}M^{g_2}$ with $\gamma(0)=p$ and $\gamma(1)=q$.
By continuity there is $s\in[0,1]$ with $\gamma(s)\in\dM$.
The restricted curves $\gamma|_{[0,s]}$ and $\gamma|_{[s,1]}$ are again length-minimizing $G_2$-geodesics.
It holds
\begin{align}\nonumber
d_{G_2}(p,q)=\ell_{G_2}(\gamma)&=\ell_{G_2}(\gamma|_{[0,s]})+\ell_{G_2}(\gamma|_{[s,1]})\\\nonumber
&=d_{G_2}(p,\gamma(s))+d_{G_2}(\gamma(s),q)\\\nonumber
&=d_{g_2}(p,\gamma(s))+d_{g_2}(\gamma(s),q)\\\nonumber
&\geq(\tfrac{1}{A}d_{g_1}(p,\gamma(s))-B)+(\tfrac{1}{A}d_{g_1}(\gamma(s),q)-B)\\\nonumber
&=(\tfrac{1}{A}d_{G_1}(p,\gamma(s))-B)+(\tfrac{1}{A}d_{G_1}(\gamma(s),q)-B)\\\nonumber
&\geq\tfrac{1}{A}d_{G_1}(p,q)-2B.
\end{align}
The other inequality $d_{G_2}(p,q)\leq A\,d_{G_1}(p,q)+2B$ can be shown in a similar manner.
\end{proof}
\begin{corollary}\label{fromHKRS08}
Let $n\geq 3$.
The right half-space $\mathbb{R}^n_{x_1\geq 0}$ cannot carry any complete Riemannian metric of uniformly positive scalar curvature which has mean convex boundary and which is quasi-isometric to the Euclidean metric via the identity.
\end{corollary}

\begin{proof}
Set $M:=\mathbb{R}^n_{x_1\geq 0}$.
Assume there is a Riemannian metric $g$ as in the theorem.
We apply Corollary~\ref{uniformization} and Theorem~\ref{theorem27} for the metric $g$ and the distinguished metric $\beta:=g_{\mathrm{eucl}}$.
This yields a Riemannian metric $f\in\mathscr{R}(M)$ with the following properties:
\begin{enumerate}
\item[\myicon]{$f$ has uniformly positive scalar curvature;}
\item[\myicon]{$f$ is doubling;}
\item[\myicon]{$\varrho_f=\varrho_{\beta}$ on a neighbourhood of $\partial M\subset[0,\infty)\times\partial M$;}
\item[\myicon]{$f$ is quasi-isometric to $g_{\mathrm{eucl}}$ via the identity;}
\item[\myicon]{$f$ is complete.}
\end{enumerate}
From the third property we deduce that $f$ and $\beta$ induce the same smooth structure on the double $\mathsf{D}M$.
Both $B=\beta\cup\beta$ and $F=f\cup f$ are smooth metrics on $\mathsf{D}M$.
The metric $F$ clearly has uniformly positive scalar curvature.
Furthermore, $F$ is complete by Theorem~\ref{completemetrics} and the identity $\id\colon(\mathsf{D}M,d_B)\to(\mathsf{D}M,d_F)$ is a quasi-isometry by Lemma~\ref{quasiisometrydouble}.

Now we apply the canonical diffeomorphism
\begin{equation}\nonumber
\kappa\colon\mathbb{R}^n\to\mathsf{D}M\,,\,\kappa(x)=\left\{\begin{array}{ll} \left[(x_1,x_2,\dots,x_n),1\right], & x_1\geq 0 \\
         \left[(-x_1,x_2,\dots,x_n),2\right], & x_1\leq 0\end{array}\right.
\end{equation}
to obtain the following commutative diagram of Riemannian manifolds and smooth maps:
\begin{equation}\nonumber
\begin{tikzcd}
(\mathsf{D}M,B)\arrow{rr}{\id}&&(\mathsf{D}M,F)\\
\\
(\mathbb{R}^n,G_{\mathrm{eucl}})\arrow{uu}{\kappa}\arrow{rr}{\id}&&(\mathbb{R}^n,\kappa^\ast F)\arrow{uu}{\kappa}
\end{tikzcd}
\end{equation}
This is where the unique smooth structure of $\mathsf{D}M$ comes into play because it guarantees smoothness for the right vertical arrow.
Both vertical arrows are Riemannian isometries.

The diagram above gives rise to a commutative diagram of metric spaces and continuous maps:
\begin{equation}\nonumber
\begin{tikzcd}
(\mathsf{D}M,d_B)\arrow{rr}{\id}&&(\mathsf{D}M,d_F)\\
\\
(\mathbb{R}^n,d_{\mathrm{eucl}})\arrow{uu}{\kappa}\arrow{rr}{\id}&&(\mathbb{R}^n,d_{\kappa^\ast F})\arrow{uu}{\kappa}
\end{tikzcd}
\end{equation}
The vertical arrows are metric isometries.
As the upper horizontal arrow is a quasi-isometry, so is the lower one.
In total, $\kappa^\ast F$ violates the non-existence result in \cite[Cor.~1.8]{HKRS}.
\end{proof}
\subsection{Adding topology}
The quasi-isometry condition can be dropped if we add some topology to the half-space.
\begin{corollary}\label{Application3}
Let $n\geq 2$ and $N:=\mathbb{R}^n_{x_1\geq 0}\#T^n$ be a connected sum of the right half-space with an $n$-torus attached in the interior.
Then $N$ does not admit a complete Riemannian metric of positive scalar curvature which has mean convex boundary.
\end{corollary}
\begin{proof}
The strategy is the same as before.
Since the double $\mathsf{D}M$ is diffeomorphic to a connected sum $(T^n\# T^n)\setminus\nobreak\{\mathrm{pt}\}$ of tori with a single point removed, and since $T^n\# T^n$ is enlargeable, we can apply \cite[Thm.~C]{Cecchini} by Cecchini.
\end{proof}
At least in low dimensions, this result is valid in greater generality, meaning that the right half-space can be replaced by any other manifold with boundary.
We refer to this as the generalised Geroch conjecture with boundary.
\begin{corollary}\label{Geroch}
Let $n\geq 2$ such that the generalised Geroch conjecture (without boundary) holds true for all manifolds of dimension $n$.
Then for any $n$-manifold $M$ with non-empty boundary, the connected sum $M\# T^n$ with a torus attached in the interior does not admit a complete Riemannian metric of positive scalar curvature which has mean convex boundary.
\end{corollary}
\begin{proof}
It holds $\mathsf{D}(M\# T^n)\approx(\mathsf{D}M\#T^n)\#T^n$.
Applying the generalised Geroch conjecture for $\mathsf{D}M\#T^n$ leads to a contradiction in the usual proof strategy.
\end{proof}
Recent progress on the generalised Geroch conjecture by Lesourd-Unger-Yau \cite[Thm.~1.2]{LUY} and Chodosh-Li \cite[Thm.~3]{CL} shows that the implication in Corollary~\ref{Geroch} is true for all $n\leq 7$.
\begin{remark}
There are non-existence results similar to Corollary~\ref{Application3} and \ref{Geroch} if the torus is cut along one of its basis circles and glued along the boundary of $\mathbb{R}^n_{x_1\geq 0}$ or $M$, respectively.
\end{remark}
Together with a recent result of Cruz-Silva Santos \cite[Thm.~1.1]{CSS} and three splitting theorems for manifolds with boundary by Kasue \cite[Thm.~B~\&~C]{Kasue} and Sakurai \cite[Thm.~1.8]{Sakurai}, we can draw the following conclusion:
\begin{corollary}\label{Gerochadd}
Let $n\geq 3$ such that the generalised Geroch conjecture (without boundary) holds true for all manifolds of dimension $n$.
\begin{enumerate}
\item[(a)]{Let $M$ be an $n$-manifold with non-empty boundary and let $M\#T^n$ be the connected sum with a torus attached in the interior.
Then every complete Riemannian metric on $M\#T^n$ with non-negative scalar curvature and mean convex boundary is Ricci-flat.}
\item[(b)]{Let $M$ be a connected $n$-manifold with non-empty boundary and one of the following topological properties:
\begin{enumerate}\nonumber
\item[(i)]{$M$ is non-compact and $\partial M$ is compact or}
\item[(ii)]{$\partial M$ is disconnected and it has a compact connected component $\Gamma$.}
\end{enumerate}
Then the connected sum $M\# T^n$ with a torus attached in the interior does not admit a complete Riemannian metric of non-negative scalar curvature which has mean convex boundary.}
\item[(c)]{For any connected $n$-manifold $M$ with non-empty and non-compact boundary, the connected sum $M\# T^n$ with a torus attached in the interior does not admit a complete Riemannian metric $g$ of non-negative scalar curvature which has mean convex boundary and which has a unit-speed normal geodesic $\gamma\colon[0,\infty)\to M$ emanating from the boundary with
\begin{equation}\nonumber
\sup\{t\geq 0: d_g(\gamma(t),\partial M)=t\}=\infty.
\end{equation}}
\end{enumerate}
\end{corollary}
The results in $(a)$ and $(c)$ hold for all $n\geq 3$ (independently of the Geroch conjecture) when $M=\mathbb{R}^n_{x_1\geq 0}$.
\begin{proof}
\textit{Regarding} \textit{(a)}:
If there existed a non-Ricci-flat complete metric of non-negative scalar curvature with mean convex boundary, then \cite[Thm.~1.1]{CSS} by Cruz-Silva Santos would yield a complete metric of positive scalar curvature with mean convex boundary on $M\#T^n$.
In fact, such a metric cannot exist according to Corollary~\ref{Geroch}.

\textit{Regarding} \textit{(b)}: Let $M$ be a connected $n$-dimensional manifold with non-empty boundary that satisfies $(i)$ or $(ii)$.
We assume that there exists a metric $g$ with the above properties.
By $(a)$ it is Ricci-flat.
Then the pair $(M\# T^n,g)$ is either isometric to the direct product $[0,\infty)\times\partial M$ (property $(i)$) or $[0,a]\times\Gamma$ (property $(ii)$).
These splitting theorems are due to Kasue \cite[Thm.~B~\&~C]{Kasue}.
The one for property $(i)$ has also been proved by Croke-Kleiner in \cite[Thm.~2]{CK}.

We focus on property $(i)$, the other case is similar:
Given an isometry $\Phi\colon(M\#T^n,g)\to([0,\infty)\times\partial M,\mathrm{d}t^2+g_0)$, one can arrange that $\Phi|_{\partial M}=\id_{\partial M}$ by eventually composing the original map with $\id_{[0,\infty)}\times\Phi|_{\partial M}^{-1}$.
Then the inclusion $\partial M\hookrightarrow M\#T^n$ is a homotopy equivalence.
In particular, the induced map in homology $H_1(\partial M)\to H_1(M\#T^n)$ is an isomorphism.

Writing $M\# T^n$ as an attachment space $(M\setminus\mathring{D}^n)\cup_{S^{n-1}}(T^n\setminus\mathring{D}^n)$, we know that the canonical map $H_1(M\setminus\mathring{D}^n)\oplus H_1(T^n\setminus\mathring{D}^n)\to H_1(M\#T^n)$ in the Mayer-Vietoris sequence is injective.
Thus, by functoriality, the map
\begin{equation}\nonumber
H_1(\partial M)\to H_1(M\setminus\mathring{D}^n)\oplus H_1(T^n\setminus\mathring{D}^n)\;\;,\;\;x\mapsto(H_1(\iota)(x),0)
\end{equation}
is surjective, where $\iota\colon\partial M\hookrightarrow M\setminus\mathring{D}^n$ denotes the inclusion.
This is absurd, since $H_1(T^n\setminus\mathring{D}^n)\cong H_1(T^n)$ has non-trivial elements.

\textit{Regarding} \textit{(c)}:
The proof is anologous to part $(b)$.
Here we need to use a splitting theorem of Sakurai \cite[Thm.~1.8]{Sakurai} for non-compact boundaries.
\end{proof}
\subsection{The Whitehead manifold}
One can also drop the quasi-isometry condition by considering certain submanifolds with boundary of the Whitehead manifold instead of half-spaces.
The Whitehead manifold $W$ is an open 3-manifold which is contractible but not homeomorphic to $\mathbb{R}^3$.
Chang-Weinberger-Yu \cite[Thm.~1]{CWY} have shown that $W$ does not admit any complete metric of uniformly positive scalar curvature.
Subsequent work of Wang \cite[Thm.~1.1]{Wang} even excludes the existence of complete metrics of positive scalar curvature without uniform positive lower bound.

\subsubsection{Construction}
The Whitehead manifold is a well-established object in geometric topology.
We will look at two equivalent pictures for the Whitehead manifold:
First consider the $3$-sphere $S^3\subset\mathbb{C}^2$ as a union of two solid tori
\begin{align}\nonumber
&T_0=\{(x_0,y_0,x_1,y_1)\in S^3:x_0^2+y_0^2\leq 1/2\},\\\nonumber
&T_1=\{(x_0,y_0,x_1,y_1)\in S^3:x_1^2+y_1^2\leq 1/2\}.
\end{align}
Then embed another solid torus $T_2$ in $T_1$ such that $T_2$ forms a Whitehead link with the (thickened) meridian of $T_1$.
Put
\begin{equation}\nonumber
A:=T_1-\Int(T_2).
\end{equation}
This is a manifold with two boundary components, where the outer component is denoted by $\partial^1$ and the inner one is denoted by $\partial^2$.
For our purposes, one needs to observe the existence of a self-diffeomorphism
\begin{equation}\nonumber
\text{$\vartheta\colon A\to A$\;\;with\;\;$\vartheta(\partial^1)=\partial^2\;,\;\vartheta(\partial^2)=\partial^1$\;\;and\;\;$\vartheta|_{\partial A}^2=\id_{\partial A}$.}
\end{equation}
To this end, we draw $T_0$ as a standard solid torus in $\mathbb{R}^3$ and wrap $T_2$ around it so that $T_0$ and $T_2$ form a Whitehead link.
Then $T_1$ ist the complement of $\Int(T_0)$ in $\mathbb{R}^3$ united with a point at infinity.
Note that the meridian of $T_1$ is the equator of $T_0$, which means that the Whitehead link with components $T_0$ and $T_2$ is the one mentioned above.
The boundary of $T_0$ is $\partial^1$ and the boundary of $T_2$ is $\partial^2$.

It is a well-known fact that there exists an isotopy $(T_0\amalg T_2)\times[0,1]\to\mathbb{R}^3$ from the inclusion to an embedding $(T_0\amalg T_2)\times\{1\}\to\mathbb{R}^3$ which swaps the positions of the tori in the Whitehead link and which acts as an involution.
This isotopy can be extended to a diffeotopy $S^3\times[0,1]\to S^3$ yielding a diffeomorphism
\begin{equation}\nonumber
\text{$\Theta\colon S^3\to S^3$\;\;with\;\;$\Theta(T_0)=T_2\;,\;\Theta(T_2)=T_0$\;\;and\;\;$\Theta|_{T_0\amalg T_2}^2=\id_{T_0\amalg T_2}$}
\end{equation}
between starting and end point.
Then $\vartheta:=\Theta|_A\colon A\to A$ does the job.

Furthermore, we define
\begin{equation}\nonumber
\Psi\colon S^3\to S^3\;,\;\Psi(x_0,y_0,x_1,y_1)=(x_1,y_1,x_0,y_0).
\end{equation}
This diffeomorphism is an involution which swaps $T_0$ and $T_1$. In particular, it holds $\Psi|_{\partial^1}^2=\id_{\partial^1}$ on $\partial^1=\partial T_1=\partial T_0$.
We set $\psi:=\Psi|_{T_1}$.
After all, the composition
\begin{equation}\nonumber
\tw:=\Theta\circ\Psi|_{T_1}\colon T_1\to T_1
\end{equation}
is a twisting map which deforms $T_1$ into $\tw(T_1)=T_2$.
Then put
\begin{equation}\nonumber
T_3:=\tw(T_2)\;,\;T_4:=\tw(T_3)\;,\;\dots\;,\;T_i:=\tw(T_{i-1})
\end{equation}
for all $i\geq 3$.
The map $\Psi$ is isotopic to the identity via
\begin{align}\nonumber
S^3\times[0,1]\to S^3\;,\;(x_0,y_0,x_1,y_1,t)\mapsto\Bigl(&\sin\Bigl(\frac{\pi}{2}t\Bigr)x_0+\cos\Bigl(\frac{\pi}{2}t\Bigr)x_1\,,\,\sin\Bigl(\frac{\pi}{2}t\Bigr)y_0+\cos\Bigl(\frac{\pi}{2}t\Bigr)y_1\,,\\\nonumber
&\cos\Bigl(\frac{\pi}{2}t\Bigr)x_0-\sin\Bigl(\frac{\pi}{2}t\Bigr)x_1\,,\,\cos\Bigl(\frac{\pi}{2}t\Bigr)y_0-\sin\Bigl(\frac{\pi}{2}t\Bigr)y_1\Bigr)
\end{align}
concatenated with
\begin{equation}\nonumber
S^3\times[0,1]\to S^3\;,\;(x_0,y_0,x_1,y_1,t)\mapsto\bigl(x_0,y_0,-\cos(\pi t)x_1+\sin(\pi t)y_1,-\sin(\pi t)x_1-\cos(\pi t)y_1\bigr).
\end{equation}
Furthermore, $\Theta$ is isotopic to $\id_{S^3}$ by construction.
Therefore, the composition $\Theta\circ\Psi$ is isotopic to $\id_{S^3}$ as well.
As a consequence, each pair of tori $T_{i-1}$ and $T_i$ is Whitehead linked.

The Whitehead manifold is defined as the open submanifold $W:=S^3-\cap_{k=1}^\infty T_k\subset S^3$. It can also be seen as the union of $T_0$ with infinitely many `peeled' tori:
\begin{equation}\nonumber
W=T_0\cup\bigl(T_1-T_2\bigr)\cup\bigl(T_2-T_3\bigr)\cup\ldots
\end{equation}
This description is the motivation for the second picture:
We consider the infinite sequence of spaces
\begin{align}\nonumber
&W_0=T_1,\\\nonumber
&W_1=A\cup_{\partial^2}\tw(W_0),\\\nonumber
&W_2=A\cup_{\partial^2}\tw(W_1),\\[-4pt]\nonumber
&\hspace{4pt}\vdots\\\nonumber
&W_i=A\cup_{\partial^2}\tw(W_{i-1})
\end{align}
together with the obvious inclusion maps $\tw\colon W_{i}\hookrightarrow W_{i+1}$.
In addition, we have a sequence of embeddings $f_i\colon W_i\to W$ given by
\begin{align}\nonumber
&f_0=\psi,\\\nonumber
&f_1=\left\{\begin{array}{ll} f_0\circ\tw^{-1}& \text{on}\;\tw(W_0) \\
         \vartheta& \text{on\;$A$}\end{array}\right.\hspace{-4pt},\\\nonumber
&f_2=\left\{\begin{array}{ll} f_1\circ\tw^{-1}& \text{on}\;\tw(W_1) \\
         \tw\circ\vartheta& \text{on\;$A$}\end{array}\right.\hspace{-4pt},\\[-4pt]\nonumber
&\hspace{4pt}\vdots\hspace{60pt}\vdots\\\nonumber
&f_i=\left\{\begin{array}{ll} f_{i-1}\circ\tw^{-1}& \text{on}\;\tw(W_{i-1}) \\
         \tw^{i-1}\circ\vartheta& \text{on\;$A$}\end{array}\right.\hspace{-4pt}.
\end{align}
These maps are well-defined since $\psi=\psi^{-1}$ on $\partial^1$ and $\vartheta=\vartheta^{-1}$ on $\partial^2$.
In a nutshell, the maps $f_i$ send the innermost torus of $W_i$ to $T_0$ and the peeled tori that are stacked to the outside of $W_i$ to the peeled tori inside $W$.

As the collection $\{T_0,T_1-\Int(T_2),T_2-\Int(T_3),\ldots\}$ of closed subsets of $W$ is locally finite, a map $W\to Y$ to any topological space $Y$ is continuous iff it is continuous on every subset in the collection.
Thus, $W\to Y$ is continuous iff $W_i\to W\to Y$ is continuous for every $i$, showing that $W$ is the (topological) direct limit of the $W_i$.
This is the second picture for the Whitehead manifold.

Both pictures give geometrical and topological insights: 
While the manifold structure of $W$ is immediately visible in the first picture, its contractibility can be deduced from the second picture relying on the fact that $W_i\to W_{i+1}$ is nullhomotopic for every $i$ and that the $\pi_n$-functors commute with direct limits. 

Now we remove the innermost torus of $W$, corresponding to the torus $T_0$ in the first picture.
As has been shown by Newman-Whitehead \cite[Thm.~6]{NW}, the fundamental group $\pi_1(W-T_0)$ is not finitely generated.
This is sufficient to show that $W$ is not homeomorphic to $\mathbb{R}^3$:
If there was a homeomorphism $\Phi\colon W\to\mathbb{R}^3$, then it would hold $\pi_1(W-T_0)\cong\pi_1(\mathbb{R}^3-\Phi(T_0))$. However, $\pi_1(\mathbb{R}^3-\Phi(T_0))$ is finitely generated (Wirtinger presentation).

\subsubsection{Creating non-compact boundaries}
Next, we turn to submanifolds with boundary of $W$.
To this end, let $R$ be a closed tube around a ray $[0,\infty)$.
One can embed $R$ properly into $W$ by means of the second picture for the Whitehead manifold, cf. Figure \ref{Whiteheadpicture}:
First cut $R$ after the interval $[0,1]$ of the ray, place the tip of the tube somewhere in $W_1-\tw(W_0)$ and let the cut tube go straight to the boundary of $W_1$ so that it misses the innermost torus $\tw(W_0)$.
Afterwards, we twist $W_1$ and add a peeled torus to obtain $W_2$.
Then take the next part of the tube that covers the interval $[1,2]$ and stick it along the first part that has already been embedded.
One may need to smooth out the glue joint.
The second part of the tube should point towards the boundary of $W_2$ as well.
This procedure can be iterated.

Once the tube $R$ has been embedded properly,
\begin{equation}\nonumber
\breve{W}:=W-\Int(R)
\end{equation}
is a submanifold of $W$ with non-compact boundary $\partial\breve{W}\approx\mathbb{R}^2$:
the required charts around $\partial R$ can be constructed using the normal bundle of $R$.
The remaining charts are immediate, since the complement of $R$ is an open subset of $W$.
Here we use the fact that the embedding $R\hookrightarrow W$ is proper.
\begin{remark}
Properness is an essential ingredient.
In general, it is not enough to properly embed only the ray $[0,\infty)$ instead of the whole tube $R$ to obtain smooth submanifolds with boundary.
For example, if we consider the manifold $\mathbb{R}^3\setminus\{0\}$ with the tube $R$ sitting directly above the $x_1$-$x_2$-plane, then $(\mathbb{R}^3\setminus\{0\})-\Int(R)$ has an edge.
Here the ray $[0,\infty)$ is properly embedded, but $R$ is not.
\end{remark}
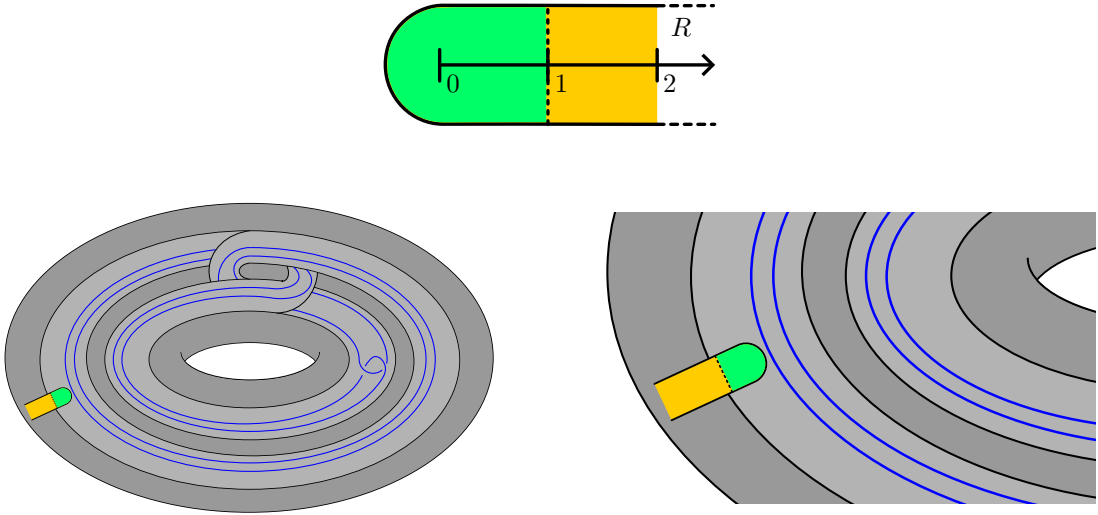
\begin{figure}[h]
\begin{tikzpicture}
\node{\includegraphics[width=0.275\textwidth]{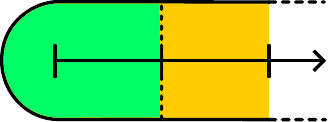}};
\node at (-1.275,-0.25) {$0$};
\node at (0.15,-0.25) {$1$};
\node at (1.575,-0.25) {$2$};
\node at (1.725,0.5) {$R$};
\end{tikzpicture}\\[0.75cm]
\begin{minipage}[c]{0.49\textwidth}
\centering
\begin{tikzpicture}
\node{\includegraphics[width=0.825\textwidth]{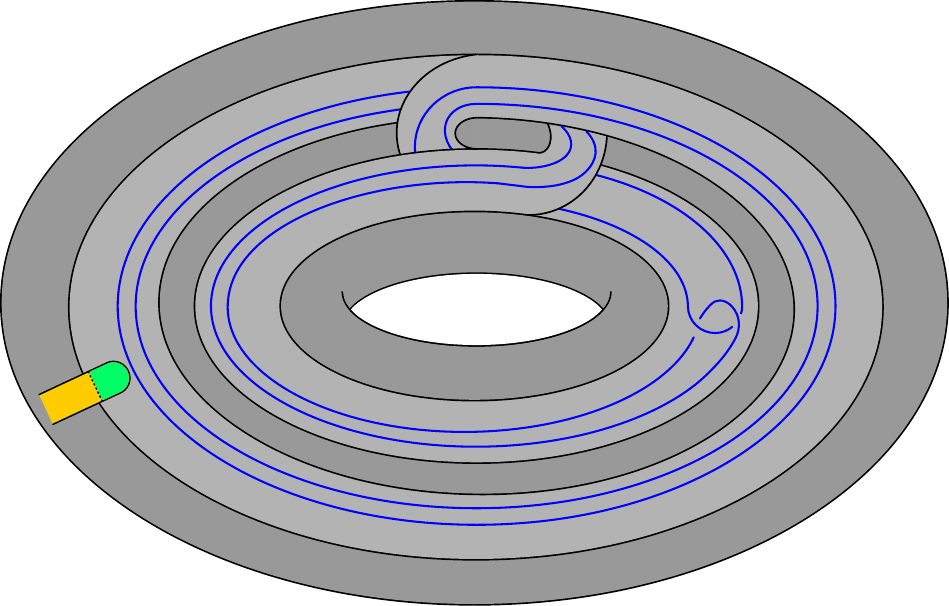}};
\end{tikzpicture}
\end{minipage}
\begin{minipage}[c]{0.49\textwidth}
\centering
\begin{tikzpicture}
\node{\includegraphics[width=0.825\textwidth]{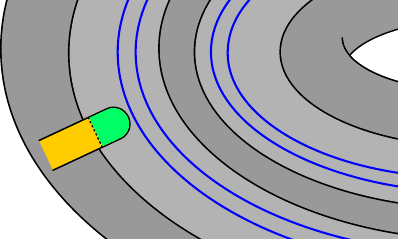}};
\end{tikzpicture}
\end{minipage}
\caption{Tube $R$ around the ray $[0,\infty)$ (above); Tori $W_0$ (blue), $W_1$ (light grey) and $W_2$ (dark grey) together with the initial part of $R$ (below)}
\label{Whiteheadpicture}
\end{figure}

Furthermore, $\breve{W}$ is contractible.
Indeed, we apply the Seifert-van Kampen theorem for the subsets\break$\breve{W}$ and $R$ of $W$ to obtain
\begin{equation}\nonumber
1=\pi_1(W)\cong\pi_1(\breve{W})\ast_{\pi_1(\partial\breve{W})}\pi_1(R)=\pi_1(\breve{W}).
\end{equation}
Thus, by the Whitehead and Hurewicz theorems, it is sufficient to show that $\breve{W}$ is homologically trivial.
The latter follows from a Mayer-Vietoris sequence for the same subsets as above.
\begin{corollary}\label{Whitehead}
$\breve{W}$ does not admit any complete Riemannian metric of uniformly positive scalar curvature which has mean convex boundary.
\end{corollary}
\begin{proof}
Using the deformation principle and \cite[Thm.~1]{CWY}, it is sufficient to show that the double $\mathsf{D}\breve{W}$ is contractible but not homeomorphic to $\mathbb{R}^3$.
The contractibility can be proven in a similar manner to that of $\breve{W}$.
In the second part of the proof, we consider the manifold
\begin{align}\nonumber
X&:=\mathsf{D}\breve{W}-\iota(T_0)\\\nonumber
&\phantom{:}\approx(W-(T_0\cup\Int(R)))\cup_{\partial\breve{W}}\breve{W}
\end{align}
where $\iota\colon\breve{W}\to\mathsf{D}\breve{W}$ is the canonical inclusion map for the first summand.
According to Seifert-van Kampen, the fundamental group of $X$ is given by
\begin{align}\nonumber
\pi_1(X)&\cong\pi_1(W-(T_0\cup R))\ast\pi_1(\breve{W})\\\nonumber
&=\pi_1(W-(T_0\cup R))\\\nonumber
&\cong\pi_1(W-T_0)
\end{align}
and hence, it is not finitely generated.
If $\mathsf{D}\breve{W}$ was homeomorphic to $\mathbb{R}^3$, then $\pi_1(X)$ would allow a finite generating system, again using the Wirtinger presentation.
\end{proof}


\end{document}